\documentclass[a4paper,11pt]{article}
\usepackage{amssymb}
\usepackage{amsthm}
\usepackage[utf8]{inputenc}
\usepackage[T1]{fontenc}
\usepackage{lmodern}
\usepackage{graphicx}
\usepackage[a4paper]{geometry}
\usepackage{url}
\usepackage{float}
\usepackage[Rejne]{fncychap}
\usepackage{amsmath}
\usepackage[colorlinks=true,urlcolor=blue,linkcolor=blue,citecolor=red]{hyperref}
\usepackage{changepage}
\usepackage{caption}
\usepackage{subcaption}
\newcommand{\R}{\mathbb{R}}
\newcommand{\M}{\mathbb{M}}
\newcommand{\Conf}{\mathrm{Conf}}

\newcommand{\vect}{\mathrm{vect}}

\newcommand{\Mink}{\mathrm{Mink}}

\newcommand{\hy}{\mathbb{H}}

\newcommand{\eu}{\widetilde{Ein}_{1,n}}
\newcommand{\eeu}{\widetilde{Ein}_{1,n-1}}

\title{Maximality of the futures of points in globally hyperbolic maximal conformally flat spacetimes}
\author{Rym SMAÏ \thanks{The author acknowledges funding by the European Research Council under ERC-Advanced grant 101095722.} \thanks{Université Côte d'Azur, Laboratoire Jean Alexandre Dieudonné, Nice, France; Email: rym.smai@univ-cotedazur.fr}}
\date{}
\begin{document}

\newtheorem{theorem}{Theorem}[section]

\newtheorem{definition}{Definition}[section]

\newtheorem{lemma}{Lemma}[section]

\newtheorem{fact}{Fact}[section]

\newtheorem{property}{Property}

\newtheorem{criterion}{Criterion}[section]

\newtheorem{remark}{Remark}[section]

\newtheorem{example}{Example}[section]

\newtheorem{proposition}{Proposition}[section]

\newtheorem{corollary}{Corollary}[section]

\maketitle

%\begin{center}
%\Large \textbf{Maximality of the futures of points in globally hyperbolic maximal conformally flat spacetimes} \footnote{The author acknowledges funding by the European Research Council under ERC-Advanced grant 101095722.}
%\end{center}

\begin{center}
%Rym SMAÏ \footnote{Laboratoire Jean Alexandre Dieudonné,
%Université Côte d'Azur,
%Nice, France.\\ Mail:
%rym.smai@univ-cotedazur.fr.\\
%ORCID: 0000-0003-0324-349X}\\
\end{center}

\paragraph*{Abstract.} Let $M$ be a globally hyperbolic conformally spacetime. We prove that the indecomposable past/future sets (abbrev. IPs/IFs) \textemdash in the sense of Penrose, \mbox{Kronheimer} and Geroch \textemdash of the universal cover of $M$ are domains of injectivity of the developing map. This relies on the central observation that diamonds are domains of injectivity of the developing map. Using this, we provide a new proof of a result of completeness by C. Rossi, which notably simplifies the original arguments. Furthermore, we establish that if, in addition, $M$ is maximal, the IPs/IFs are maximal as globally hyperbolic conformally flat spacetimes. More precisely, we show that they are conformally equivalent to regular domains of Minkowski spacetime as defined by F. Bonsante.

\section{Introduction}

% Réécrire la partie future de points de l'intro: plan du développement:

% Les diamants sont des assiettes, ce sont les analogues des boules dans le contexte Riemannien. La tentative d'analogie avec Kulkarni-Pinkall pour construire de "grandes assiettes" échoue: l'intersection de deux diamants de Ein n'est pas toujours connexe. 

We study conformally flat structures on globally hyperbolic spacetimes. The aim of this paper is to hightlight and describe relevant domains of injectivity for the developing map. Building on this, we provide a new proof, using simple arguments, of a result of completeness, established by C. Rossi in her thesis.

\subsection{Conformally flat structures on globally hyperbolic spacetimes} \label{sec intro: conformally flat structures}

A spacetime $(M,g)$ is said to be \emph{conformally flat} if it is locally conformal to Minkowski spacetime. Such a spacetime is naturally equipped with a conformal structure, defined by the conformal class of the metric $g$. A central example of conformally flat spacetimes is the so-called \emph{Einstein universe}, denoted by $Ein_{1,n-1}$. This is the conformal \mbox{compactification} of Minkowski spacetime. In this regard, Einstein universe can be seen as the Lorentzian analogue of the conformal sphere. Its group of conformal transformations is $O(2,n)$. By a Lorentzian version of Liouville's theorem, conformally flat spacetimes of dimension $n \geq 3$ are exactly those equipped with a geometric structure locally modeled on the Einstein universe. In other words, a conformally flat structure on a spacetime is equivalent to the data of an atlas of charts taking their values in $Ein_{1,n-1}$ and whose transition maps are restrictions of elements of $O(2,n)$. 

Similarly to the Riemannian setting, the Lorentzian models of constant curvature \textemdash namely Minkowski spacetime $\R^{1,n-1}$, de Sitter spacetime $dS_{1,n-1}$ and anti-de Sitter spacetime $AdS_{1,n-1}$ \textemdash are conformally equivalent to homogeneous open subsets of the Einstein universe. Therefore, spacetimes of constant curvature are conformally flat.

We focus our study on \emph{globally hyperbolic} spacetimes. This is a natural assumption regarding the causal structure of spacetimes arising from general relativity. According to a classical result due to Geroch, a spacetime is \emph{globally hyperbolic} (abbrev. GH) if there exists a Riemannian hypersurface which is met by every inextensible causal curve exactly once; such a hypersurface is called \emph{a Cauchy hypersurface}. The global hyperbolicity imposes restrictions on the topology of the spacetime. Indeed, a globally hyperbolic spacetime is diffeomorphic to the product of a Cauchy hypersurface and $\R$, and therefore is never compact. To conduct a reasonable study, we restrict ourselves to \emph{maximal} globally hyperbolic spacetimes. A globally hyperbolic spacetime is said to be \emph{maximal} (abbrev. GHM) if it cannot be embedded in a bigger globally hyperbolic spacetime that shares a Cauchy hypersurface with it. We will discuss this notion more precisely in Section \ref{sec intro: future of points}.

A natural procedure for constructing globally hyperbolic maximal conformally flat spacetimes is to consider quotients $\Omega / \Gamma$, where $\Omega$ is an open subset of the universal cover of $Ein_{1,n-1}$, and $\Gamma$ a discrete group of conformal transformations acting freely and properly discontinuously on $\Omega$. Such conformally flat structures are said to be \emph{Kleinian}; we highlight examples arising from Anosov representations in \cite{smai2022anosov}. In particular, when $\Omega$ is the entire space, the conformally flat structure is said to be \emph{complete}. 

In \cite[Theorem 10]{Salvemini2013Maximal}, C. Rossi established a sufficient condition for a globally hyperbolic maximal conformally flat spacetime to be complete, which involves the concept of \emph{conjugate points}. In Section \ref{sec: canonical neighborhoods}, we present a short and elementary proof of this result that notably simplifies the original arguments. We will explore this topic in greater detail in Section \ref{sec intro: domains of injectivity}.  

Before that, let us provide further precision on the conformally flat Lorentzian structures in dimension $n \geq 3$. We mentionned that these structures coincide with $(O(2,n), Ein_{1,n-1})$-structures. It is then a classical fact that they are encoded by the data of a local diffeomorphism $D$ from the universal cover $\tilde{M}$ of $M$ to $\eeu$, called \emph{a developing map}, and a group morphism $\rho$ from the fundamental group of $M$ to $O(2,n)$ equivariant with respect to $D$, called \emph{the holonomy representation}. The developing map encodes the data of an atlas of charts on $M$ taking their values in $Ein_{1,n-1}$. The transition maps, for their parts, are given by the holonomy morphism $\rho$. 

In general, the developing map $D$ is neither injective nor surjective. The case where it is injective (resp. a global diffeomorphism) corresponds to the Kleinian (resp. complete) structures. The goal of this paper is to highlight relevant open subsets of $\tilde{M}$ on which the developing map is injective. Such open subsets will be called \emph{domains of injectivity} of the developing map.

\subsection{Completeness of GHM conformally flat spacetimes}

Given a simply-connected, globally hyperbolic conformally flat spacetime $M$ of dimension $n$, the first natural candidates for domains of injectivity of the developing map are \textbf{the diamonds}. We call diamond any intersection between the causal past of a point $p$ and the causal future of a point $q$, where $p$ and $q$ are two chronologically-related points of $M$ such that $p$ lies in the future of $q$. This intersection is denoted $J(p,q)$. It turns out that the diamonds of a globally hyperbolic spacetime are \emph{compact}. In various respects, diamonds appear as the Lorentzian analogues of closed balls in the Riemannian setting. 

Using elementary topological arguments, we prove the following lemma, which turns out to be central to establishing larger domains of injectivity.

\begin{lemma}[Injectivity on diamonds] \label{lemma intro: diamonds}
The restriction of the developing map to any diamond $J(p,q)$ of $M$ is injective. Moreover, its image is exactly the diamond $J(D(p), D(q))$ of $\eeu$.
\end{lemma}

This lemma allows us to give an elementary and short proof of the following statement established by Rossi:

\begin{theorem}[{\cite[Theorem 10]{Salvemini2013Maximal}}] \label{thm intro: conjugate points}
If $M$ has conjugate points then $M$ admits a \mbox{$(n-1)$-} topological sphere as a Cauchy hypersurface and therefore, the developing map is a diffeomorphism on its image.
\end{theorem}

Two points $p$ and $q$ of $M$ are said to be \emph{conjugate} if there exist two distinct lightlike geodesics connecting them. When $M$ is maximal, Theorem \ref{thm intro: conjugate points} immediatly implies that $M$ is conformally equivalent to $\eeu$. 

When $M$ is maximal but not simply-connected, Theorem \ref{thm intro: conjugate points} says that the existence of conjugate points in the universal cover of $M$ is a sufficient condition of completeness. In this case, the fundamental group of $M$ preserves a compact hypersurface and therefore, is finite. We deduce the following corollary.

\begin{corollary} \label{cor intro: elliptic}
Let $M$ be a globally hyperbolic maximal conformally flat spacetime. If the universal cover of $M$ admits conjugate points, then $M$ is a finite quotient of $\eeu$.
\end{corollary} 

By analogy with the Riemannian setting, a conformally flat spacetime which is conformally equivalent to a finite quotient of $\eeu$ will be said to be \emph{elliptic}.

\subsection{Domains of injectivity of the developing map} \label{sec intro: domains of injectivity}

According to Corollary \ref{cor intro: elliptic}, we will henceforth consider simply-connected, globally hyperbolic conformally flat spacetimes that do not have conjugate points. Building on the fact that the diamonds are domains of injectivity for the developing map, our goal is to highlight \emph{larger} domains of injectivity.

\subsubsection{A natural approach that fails: the canonical neighborhoods}

A natural approach consists of mimicking the construction of Kulkarni and Pinkall in the context of conformally flat Riemannian manifolds. The authors define the canonical neighborhood of a point $p$ in a simply-connected conformally flat Riemannian manifold as the union of all open balls containing $p$, and they prove that these canonical neighborhoods are domains of injectivity for the developing map. As mentioned above, diamonds are suitable analogues of balls in the Lorentzian setting. Therefore, one can define the canonical neighborhood of a point $p$ in a simply-connected, globally hyperbolic conformally flat spacetime $M$ as the union of the interiors of diamonds containing $p$.

Since the developing map is injective on diamonds, a sufficient condition for it to be injective on the canonical neighborhood of $p$ is that the intersection of the images, under the developing map, of the interiors of any two diamonds in $M$ is connected (see the classical \emph{Lemme des assiettes} \ref{lemme des assiettes}). The image under the developing map of the interior of a diamond in $M$ is the interior of a diamond in $\eeu$ (see Corollary \ref{cor: D is injective on interior of diamonds}). Therefore, what we need to verify is that the intersection of the interiors of two diamonds in $\eeu$ is connected. However, it turns out that this is not always the case. In Section \ref{sec: intersection of diamonds}, we present an example where the intersection is disconnected. The failure of this connectedness allows to construct examples of globally hyperbolic conformally flat spacetimes where the developing map is not injective on the canonical neighborhood of a point. Consequently, the naive analogy of the canonical neighborhoods fails in the Lorentzian setting. At this point, we still don't know what a good analogy would be. Nevertheless, the injectivity of the developing map on the diamonds allows to establish the injectivity on other natural domains regarding the causal structure: \emph{the futures and the past of points}.

\subsubsection{Futures/pasts of points in a non-elliptic GHM conformally flat spacetime}\label{sec intro: future of points}

In Section \ref{sec intro: future of points}, we prove that the futures \textemdash and by symmetry the pasts \textemdash of points are domains of injectivity for the developing map. This follows directly from our central Lemma \ref{lemma intro: diamonds}. The key idea is that the chronological future of a point $p$ in $M$ can be described as the union of the interiors of diamonds $J(p,q)$, where $q$ is a point in $M$ that lies in the chronological future of $p$. Using again the classical \emph{Lemme des assiettes}, we verify that the the intersection of the interiors of any two diamonds of $\eeu$ sharing the same past vertex, is connected.

We deduce from the injectivity of the developing map on the futures of points that it is also injective on the chronological future of any past-inextensible causal curve of~$M$. These domains are significant, as they correspond to the terminal indecomposable future sets (abbrev. TIFs) described by Kronheimer, Penrose and Geroch in their paper \cite{Kronheimer}, defining what these authors call \emph{the past causal boundary of~$M$}. By symmetry, the chronological past of any future-inextensible causal curve is also a domain of injectivity for the developing map, corresponding to the terminal indecomposable past sets (abbrev. TIPs) and defining \emph{the future causal boundary of $M$}. The futures (resp. the pasts) of points and the TIFs (resp. TIPs) can be grouped under the unified concept of indecomposable future (resp. past) sets (abbrev. IFs (resp. IPs)). We provide a brief overview of these concepts in Section \ref{sec: causal completion}.

\subsection{Maximality of the IPs/IFs in a non-elliptic GHM conformally flat spacetime} \label{sec intro: maximality}

Let $M$ be a simply-connected, globally hyperbolic conformally flat spacetime that do not have conjugate points. A relevant property of the IFs/IPs is that they are globally hyperbolic. Let us now define more precisely the notion of maximality for globally hyperbolic conformally flat spacetimes introduced in Section \ref{sec intro: conformally flat structures}. There is a natural order relation on globally hyperbolic conformally flat spacetimes defined as follows. 

Given two globally hyperbolic conformally flat spacetimes, $M$ and $N$, we say that $N$ is \emph{a Cauchy-extension} of $M$ if there exists a conformal embedding $f$ from $M$ to $N$ that maps any Cauchy hypersurface of $M$ on a Cauchy hypersurface of $N$. The map $f$ is called \emph{a conformal Cauchy-embedding}. The spacetime $M$ is said to be \emph{maximal} if any conformal Cauchy-embedding from $M$ to any conformally flat globally hyperbolic spacetime $N$ is surjective. Using Zorn lemma, C. Rossi proved in \cite[Sec. 3]{Salvemini2013Maximal} that any globally hyperbolic conformally flat spacetime admits a conformally flat maximal extension, which is unique up to conformal diffeomorphisms. We provide a constructive proof of this result in \cite[Sec. 4 and 5]{smai2023enveloping}.

Since the IFs/IPs are globally hyperbolic, it is natural to inquire whether they are maximal. It is important to emphasize that the notion of maximality we are discussing depends on the Cauchy hypersurfaces of the considered spacetime. Therefore, the ambiant spacetime $M$ is, in general, not a maximal extension of the chronological future of a point $p$, as their respective Cauchy hypersurfaces are completely independent. In fact, we establish that the IFs/IPs are maximal if the ambiant spacetime $M$ is itself maximal. 

\begin{theorem} \label{thm intro: maximality}
Any IF/IP of $M$ is a \textbf{maximal} globally hyperbolic conformally flat spacetime.
\end{theorem}

This result was first established by C. Rossi in her thesis (see \cite[Chap. 6, Prop. 3.6 \& Thm. 3.9]{salveminithesis}). In fact, she separately proves that the futures of points are maximal (see \cite[Chap. 6, Prop. 3.6]{salveminithesis}), and then that the TIFs are also maximal (see \cite[Chap. 6, Thm. 3.9]{salveminithesis}). In Section \ref{sec: maximality of IPs}, we propose a synthetic proof that establishes both of these results simultaneously. Our proof relies on the following property of the maximal extensions that we proved in a previous paper and which says, in short, that \emph{the functor maximal extension} is compatible with the inclusion: 

\begin{fact}[{\cite[Theorem 5]{smai2023enveloping}}]\label{fact intro}
Let $V$ be a globally hyperbolic conformally flat spacetime and let $U$ be a causally convex open subset of $V$. Then, the maximal extension of $U$ is conformally equivalent to a causally convex open subset of the maximal extension of $V$.
\end{fact}

In fact, the proof of Theorem \ref{thm intro: maximality} shows more precisely that the IFs (resp. IPs) of $M$ are conformally equivalent to past (resp. future) regular domains of Minkowski spacetime~$\M^n$ (see Definition \ref{def: regular domains}). These are past-complete (resp. future-complete) \emph{convex} open subsets of $\M^n$. We prove that the domains corresponding to \emph{the futures (resp. pasts) of points} are \emph{strictly} convex in the following sense:

\begin{theorem} \label{thm intro: stricly convex}
The chronological future (resp. past) of a point $p$ in $M$ is conformally equivalent to a past-regular (resp. future-regular) domain of Minkowski spacetime which does not contain any spacelike segment in its boundary.
\end{theorem}

This result does not hold for the TIPs/TIFs.

\subsection*{Overview of the paper} 

Section \ref{sec: conformal compactification of Minkowski spacetime} presents Einstein universe as the conformal compactification of Minkowski spacetime. In particular, we provide a detailed description of the conformal boundary of Minkowski spacetime, known as \emph{the Penrose boundary}. In Section \ref{sec: regular domains}, we introduce the notion of \emph{regular domains} and we characterize them in Einstein universe using the notion of \emph{shadows} developed by C. Rossi. Section \ref{sec: conformally flat spacetimes} reviews some well-know facts about globally hyperbolic conformally flat spacetimes and their causal boundaries. In Section \ref{sec: canonical neighborhoods}, we prove that the diamonds are domains of injectivity for the developing map and we establish Theorem \ref{thm intro: conjugate points}. We also prove that the IPs/IFs are domains of injectivity for the developing map. Last but not least, we establish the maximality of the IPs/IFs and we prove Theorem \ref{thm intro: stricly convex} in Section \ref{sec: maximality of IPs}. 

\subsection*{Acknowledgement}

I would like to express my gratitude to my PhD advisor, Thierry Barbot, for his time, guidance, thoughtful insights, and the many engaging discussions that made this work possible. I am also thankful to Adam Chalumeau for pointing out that the intersection of two diamonds in the Einstein universe might not always be connected and for the stimulating conversations on the geometry of these diamonds, which ultimately shaped Sections~\ref{sec: diamonds in the Einstein universe} and \ref{sec: intersection of diamonds}.

\section{Conformal compactification of Minkowski spacetime} \label{sec: conformal compactification of Minkowski spacetime}

We assume the reader sufficiently acquainted with causality of spacetimes, namely the notions of \emph{causal curves}, \emph{future and past of a subset}, \emph{diamond}, \emph{lightcones}, \emph{achronal and acausal subsets}, \emph{global hyperbolicity}, etc. We direct to \cite[Chap. 14]{oneill} for further details.

Throughout this document, given two subsets $A$ and $U$ of a spacetime $M$ such that $A \subset U$, we denote by $J^\pm(A,U)$ (resp. $I^\pm(A,U)$) the causal (resp. chronological) future/past of $A$ relatively to $U$. When $U = M$, we simply write $J^\pm(A)$ (resp. $I^\pm(A)$). A diamond $J^+(q) \cap J^-(p)$ will be denoted $J(p,q)$. Its interior, equal to $I^+(q) \cap I^-(p)$, will be denoted $I(p,q)$.\\

This section presents the conformal compactification of the flat Lorentzian model \emph{Minkowski spacetime}. In Riemannian geometry, the conformal compactification of the Euclidean space $\mathbb{E}^n$ is the conformal sphere $\mathbb{S}^n$. By analogy, the conformal compactification of Minkowski spacetime $\M^n$ is the Lorentzian analogue of the conformal sphere, the so-called \emph{Einstein universe} $Ein_{1,n-1}$.

\subsection{Einstein universe} \label{sec: Einstein universe}

In this section, we recall the geometry and the causal structure of Einstein universe.

\subsubsection{Klein model} 

Let $\R^{2,n}$ be the vector space $\R^{n+2}$ of dimension $(n+2)$ equipped with the nondegenerate quadratic form $q_{2,n}$ of signature $(2,n)$ given by
\begin{align*}
q_{2,n}(u,v,x_1,\ldots, x_n) &= -u^2 - v^2 + x_1^2 + \ldots + x_n^2
\end{align*}
in the coordinate system $(u,v,x_1,\ldots,x_n)$ associated to the canonical basis of $\R^{n+2}$.

\begin{definition}
\emph{Einstein universe} of dimension $n$, denoted by $\mathsf{Ein}_{1,n-1}$, is the space of isotropic lines of $\R^{2,n}$ with respect to the quadratic form $q_{2,n}$, namely
\begin{align*}
\mathsf{Ein}_{1,n-1} &= \{[x] \in \mathbb{P}(\R^{2,n}):\ q_{2,n}(x) = 0\}.
\end{align*}
\end{definition}

In practice, it is more convenient to work with the double cover of the Einstein universe, denoted by $Ein_{1,n-1}$:
\begin{align*}
Ein_{1,n-1} &= \{[x] \in \mathbb{S}(\R^{2,n}):\ q_{2,n}(x) = 0\}
\end{align*}
where $\mathbb{S}(\R^{2,n})$ is the sphere of rays, namely the quotient of $\R^{2,n} \backslash \{0\}$ by positive homotheties.

\subsubsection{Spatio-temporal decomposition} 

A plane $P \subset \R^{2,n}$ is said to be \emph{timelike} if the restriction of $q_{2,n}$ to $P$ is negative definite.

\begin{lemma}
The choice of a timelike plane $P \subset \R^{2,n}$ defines a diffeomorphism from $\mathbb{S}^{n-1} \times \mathbb{S}^1$ to $\eeu$.
\end{lemma}

\begin{proof}
Indeed, consider the orthogonal splitting $\R^{2,n} = P^\perp \oplus P$ and call $q_{P^\perp}$ and $q_{P}$ the positive definite quadratic form induced by $\pm q_{2,n}$ on $P^\perp$ and $P$ respectively. The restriction of the canonical projection $\R^{2,n} \backslash \{0\}$ on $\mathbb{S}(\R^{2,n})$ to the set of points $(x,y) \in P^\perp \oplus P$ such that $q_{P^\perp}(x) = q_P(y) = 1$ defines a smooth map from $\mathbb{S}^{n-1} \times \mathbb{S}^1$ to $Ein_{1,n-1}$. It is easy to see that this is a diffeomorphism.
\end{proof}

For every timelike plane $P \subset \R^{2,n}$, the quadratic form $q_{2,n}$ induces a Lorentzian metric $g_P$ on $\mathbb{S}^{n-1} \times \mathbb{S}^1$ given by
\begin{align*}
g_P &= d\sigma^2 (P) - d\theta^2(P)
\end{align*}
where $d\sigma^2(P)$ is the round metric on $\mathbb{S}^{n-1} \subset (P^\perp,q_{P^\perp})$ induced by $q_{P^\perp}$  and $d\theta^2(P)$ is the round metric on $\mathbb{S}^{1} \subset (P,q_P)$ induced by $q_P$.

An easy computation shows that if $P' \subset \R^{2,n}$ is another timelike plane, the Lorentzian metric $g_{P'}$ is conformally equivalent to $g_P$, i.e. $q_P$ and $g_P'$ are proportionnal by a positive smooth function on $\mathbb{S}^{n-1} \times \mathbb{S}^1$. As a result, Einstein universe is a conformal spacetime.

\begin{definition}
We call \emph{spatio-temporal decomposition} of $Ein_{1,n-1}$ any conformal diffeomorphism from $\mathbb{S}^{n-1} \times \mathbb{S}^1$ to $Ein_{1,n-1}$.
\end{definition}

The causal structure of Einstein universe is trivial: any point is causally related to any other one (see e.g. \cite[Cor. 2.10, Chap. 2]{salveminithesis}).

\subsubsection{Lightlike geodesics, lightcones and conformal spheres}

We characterize some remarkable geometrical objects in Einstein universe:

\begin{enumerate}
\item \emph{A photon} of $Ein_{1,n-1}$ is the intersection of $Ein_{1,n-1}$ with the projectivization of a totally isotropic plane of $\R^{2,n}$ (see e.g. \cite[Chap.  2, Lemme 2.12]{salveminithesis}).
\item \emph{A lightcone} of a point $\mathrm{x} = [x] \in Ein_{1,n-1}$ is the intersection of $Ein_{1,n-1}$ with the projectivization of the orthogonal of $x$ with respect to $q_{2,n}$. Topologically, it is a double pinched torus. Remark that in $\mathsf{Ein}_{1,n-1}$, the lightcone of a point is a pinched torus.
\item \emph{A conformal $(k-1)$-sphere} of $Ein_{1,n-1}$ is a connected component of the intersection of $Ein_{1,n-1}$ with the projectivization of a Lorentzian $(k+1)$-plane of $\R^{2,n}$.
\end{enumerate}

\begin{figure}[h!]
\centering
\begin{tabular}{ccccc}
\includegraphics[scale=1.4]{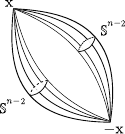} &  &  & &
\includegraphics[scale=1.5]{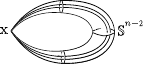}
\end{tabular}
\caption{Lightcone of a point $\mathrm{x}$ in $Ein_{1,n-1}$ (on the left) and in $\mathsf{Ein}_{1,n-1}$ (on the right).}
\end{figure}

\subsubsection{Universal Einstein universe} \label{sec: universal Einstein}

Let $\pi: \eeu \to Ein_{1,n-1}$ be the universal cover of Einstein universe. 

When $n \geq 3$, every diffeomorphism from $Ein_{1,n-1}$ and $\mathbb{S}^{n-1} \times \mathbb{S}^1$ lifts to a diffeomorphism from $\eeu$ to $\mathbb{S}^{n-1} \times \R$. The pull-back by the projection $\mathbb{S}^{n-1} \times \R \to \mathbb{S}^{n-1} \times \mathbb{S}^1$ of the conformal class of the Lorentzian metric $d\sigma^2 - d\theta^2$ on $\mathbb{S}^{n-1} \times \mathbb{S}^1$ defined previously, is the conformal class of the Lorentzian metric $d\sigma^2 - dt^2$ where $dt^2$ is the usual metric on~$\R$. This induces a natural conformal structure on $\eeu$. 

\begin{definition}
We call \emph{spatio-temporal decomposition of $\eeu$} any conformal diffeomorphism between $\eeu$ and $\mathbb{S}^{n-1} \times \R$. 
\end{definition}

In what follows, we fix a spatio-temporal decomposition and we identify $\eeu$ to $\mathbb{S}^{n-1} \times \R$. The fundamental group of $Ein_{1,n-1}$ is isomorphic to $\mathbb{Z}$, generated by the transformation $\delta: \eeu \to \eeu$ defined by $\delta(x,t) = (x, t + 2\pi)$. The fundamental group of $\mathsf{Ein}_{1,n-1}$ is generated by the transformation $\sigma: \eeu \to \eeu$ such that $\sigma^2 = \delta$, i.e. the map defined by $\sigma(x,t) = (-x, t + \pi)$.

\begin{definition}
Two points $p$ and $q$ of $\eeu$ are said to be \emph{conjugate} if one is the image under $\sigma$ of the other.
\end{definition}

\begin{remark}
If $p, q \in \eeu$ are conjugate, then $\pi(p) = - \pi(q)$.
\end{remark}

Unlike $Ein_{1,n-1}$, the causal structure of $\eeu$ is far from being trivial. We \mbox{describe} it briefly below. We direct to \cite[Chap. 2]{salveminithesis} for more details.\\

Lightlike geodesics of $\eeu$ are the curves which can be written, up to reparametrization, as $(x(t), t)$ where $x: I \to \mathbb{S}^{n-1}$ is a geodesic of $\mathbb{S}^{n-1}$ defined on an interval $I$ of $\R$. The inextensible ones are those for which $x$ is defined on $\R$. 

The photons going through a point $(x_0, t_0)$ have common intersections at the points $\sigma^k(x_0,t_0)$, for $k \in \mathbb{Z}$; and are pairwise disjoint outside these points. The following definition introduces the concept of \emph{complete} lightlike geodesics, which will be useful for the discussion ahead in Section \ref{sec: diamonds}.

\begin{definition}
A lightlike geodesic of $\eeu$ is said to be \emph{complete} if it connects two conjugate points \emph{strictly}.
\end{definition}

The lightcone of a point $(x_0, t_0)$ is the set of points $(x,t)$ such that $d(x,x_0) = |t - t_0|$ where $d$ is the distance on the sphere $\mathbb{S}^{n-1}$ induced by the round metric. It disconnects $\eeu$ in three connected components:
\begin{itemize}
\item The \emph{chronological future} of $(x_0, t_0)$: this is the set of points $(x,t)$ of $\mathbb{S}^{n-1} \times \R$ such that $d(x,x_0) < t - t_0$.
\item The \emph{chronological past} of $(x_0, t_0)$: this is the set of points $(x,t)$ of $\mathbb{S}^{n-1} \times \R$ such that $d(x,x_0) < t_0 - t$.
\item The set of points non-causally related to $(x_0, t_0)$, i.e. the set of points $(x,t)$ of $\mathbb{S}^{n-1} \times \R$ such that $d(x,x_0) > |t - t_0|$. 
\end{itemize}

The universal cover $\eeu$ is globally hyperbolic: any sphere $\mathbb{S}^{n-1} \times \{t\}$, where $t \in \R$, is a Cauchy hypersurface.

\subsubsection{Conformal group}

The subgroup $O(2,n) \subset Gl_{n+2}(\R)$ preserving $q_{2,n}$, acts conformally on $Ein_{1,n-1}$. When $n \geq 3$, the conformal group of $Ein_{1,n-1}$ is \emph{exactly} $O(2,n)$. This is a consequence of the following result, which is an extension to Einstein universe, of a classical theorem of Liouville in Euclidean conformal geometry (see e.g. \cite{francesarticle}):

\begin{theorem}\label{Liouville theorem}
Let $n \geq 3$. Any conformal transformation between two open subsets of $Ein_{1,n-1}$ is the restriction of an element of $O(2,n)$.
\end{theorem}

It is a classical fact that every conformal diffeomorphism of $Ein_{1,n-1}$ lifts to a conformal diffeomorphism of $\eeu$. Conversely, by Theorem \ref{Liouville theorem}, every conformal transformation of $\eeu$ defines a unique conformal transformation of the quotient space $Ein_{1,n-1} = \eeu / <\delta>$.

Let $\Conf(\eeu)$ denote the group of conformal transformations of $\eeu$. Let $j: \Conf(\eeu) \longrightarrow O(2,n)$ be the natural projection. This is a surjective group morphism whose kernel is generated by $\delta$. 

\subsection{Affine charts} \label{sec: affine charts}

In this section, we show that Einstein universe is the compactification of Minkowski spacetime. In the Riemannian setting, the Euclidean space is conformally equivalent to the complement of a point in the conformal sphere. In the Lorentzian setting, we should take into account causality which is latent in the Riemannian setting. Therefore, it turns out that Minkowski spacetime is conformally equivalent to the complement of \emph{a lightcone} in Einstein universe. Indeed, let $\mathrm{x} \in \mathsf{Ein}_{1,n-1}$. We denote by $M(\mathrm{x})$ the complement of the lightcone of $\mathrm{x}$ in $\mathsf{Ein}_{1,n-1}$:
\begin{align*}
M(\mathrm{x}) &= \{\mathrm{y} \in \mathsf{Ein}_{1,n-1}:\ <x,y>_{2,n} \not = 0\ \mathrm{with}\ \mathrm{x} = [x],\  \mathrm{y} = [y]\}.
\end{align*}

For every representant $x \in \R^{2,n}$ of $\mathrm{x}$, we call $f_x: M(\mathrm{x}) \times M(\mathrm{x}) \to x^\perp/\vect(x)$ the map defined by $f(\mathrm{y}, \mathrm{z}) = [y - z]$ where $y$ and $z$ are representant of $\mathrm{y}$ and $\mathrm{z}$ respectively such that $<y,x>_{2,n} = <z,x>_{2,n} = -1/2$.

\begin{lemma}\label{lemma: affine chart}
The map $f_x$ defines an affine structure on $M(\mathrm{x})$ of direction $x^\perp/\vect(x)$. Moreover, if $x, x' \in \R^{2,n}$ are two distinct representant of $\mathrm{x}$, there is $\lambda \in \R^*$ such that $f_{x'} = \lambda f_x$. \qed
\end{lemma}

The orthogonal of $x$ is degenerate. The kernel is the vector line in the direction of~$x$. A supplement of $\vect(x)$ in $x^\perp$ is a subspace such that the restriction of $q_{2,n}$ is of signature $(1,n-1)$. Therefore, $q_{2,n}$ induces a quadratic form of signature $(1,n-1)$ on the quotient $x^\perp/\vect(x)$. We get the following statement:

\begin{proposition}\label{prop: affine chart}
The open domain $M(\mathrm{x})$ is a conformal Minkowski spacetime. \qed
\end{proposition}

\begin{definition}
We call \emph{affine chart} of $\mathsf{Ein}_{1,n-1}$ any open domain of the form $M(\mathrm{x})$.
\end{definition}

In the double cover $Ein_{1,n-1}$, the complement of the lightcone of a point $\mathrm{x}$ is the disjoint union of two conformal Minkowski spacetimes $M(\mathrm{x})$ and $M(-\mathrm{x})$ where $M(\mathrm{x})$ is the open subset of $Ein_{1,n-1}$ given by
\begin{align*}
M(\mathrm{x}) &= \{\mathrm{y} \in Ein_{1,n-1}:\ <x, y>_{2,n} < 0\ \mathrm{with}\ [x] = \mathrm{x},\ [y] = \mathrm{y}\}.
\end{align*}

In the universal Einstein universe, we show that a point $p$ defines three affine charts. Let $\Mink_0(p)$ be the set of points which are not causally related to $p$:
\begin{align*}
\Mink_0(p) &= \eeu \backslash (J^+(p) \cup J^-(p)).
\end{align*}
Notice that $\Mink_0(p)$ corresponds exactly to the diamond $I(\sigma(p), \sigma^{-1}(p))$ (see Figure \ref{fig: affine charts in ein univ}).

\begin{lemma}\label{lemma: affine chart in ein univ}
The restriction of the projection $\pi: \eeu \to Ein_{1,n-1}$ to $\Mink_0(p)$ is injective. Moreover, its image is equal to the affine chart $M(\mathrm{x})$ with $\mathrm{x} = \pi(p)$.
\end{lemma}

\begin{proof}
The set $\Mink_0(p) = I(\sigma(p), \sigma^{-1}(p))$ does not contain conjugate points. Then, the restriction of $\pi$ to $\Mink_0(p)$ is injective. The fact that the image under $\pi$ of $\Mink_0(p)$ is equal to $M(\mathrm{x})$ follows immediately from \cite[Lemma 10.13]{andersson2012}.
\end{proof}

Lemma \ref{lemma: affine chart in ein univ} motivates the following definition.

\begin{definition}
We call \emph{affine chart} of the universal Einstein universe any open domain $\Mink_0(p)$.
\end{definition}

Notice that a point $p \in \eeu$ defines three distinct affine charts:
\begin{enumerate}
\item $\Mink_0(p)$, which is the set of points non-causally related to $p$;
\item $\Mink_+(p) := \Mink_0(\sigma(p))$, contained in the chronological future of $p$;
\item $\Mink_-(p) := \Mink_0(\sigma^{-1}(p))$, contained in the chronological past of $p$.
\end{enumerate}

Remark that the photons of an affine chart are all complete lightlike geodesics of $\eeu$ (see Figure \ref{fig: affine charts in ein univ}).

\begin{figure}[h!]
\centering
\includegraphics[scale=1]{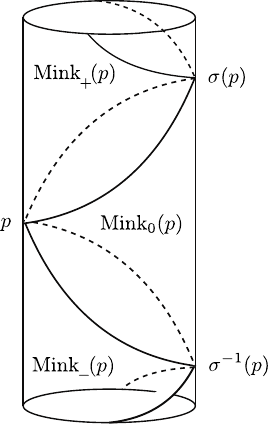}
\caption{Affine charts defined by $p \in \eeu$ for $n = 2$.}
\label{fig: affine charts in ein univ}
\end{figure}

\subsection{Diamonds in the universal Einstein universe} \label{sec: diamonds in the Einstein universe}

In the universal Einstein universe, diamonds fall into two categories: those that contain \emph{conjugate points} in their interior and those that do not. Observe that the latter are necessarily contained in some affine chart. They split into two categories: those that contain \emph{complete lightlike geodesics} in their interior and those that do not. 

Up to conformal diffeomorphism, we distinguish two types of diamonds whose interior does not contain conjugate points but contains complete lightlike geodesics:

\begin{enumerate}
\item \textbf{the affine charts}: they are conformally equivalent to Minkowski spacetime;
\item \textbf{the half-spaces bounded by an affine degenerate hyperplane in an affine chart}: they all have the same conformal structure. Indeed, let $J(p,q)$ and $J(p',q')$ two such diamonds. Up to replacing $J(p',q')$ by its image under a conformal diffeomorphism of $\eeu$ sending $p'$ on $p$, we can assume that $p = p'$. We look then for an element of the stabilizer of $p$ in the group of conformal of diffeomorphisms of $\eeu$ that sends $q'$ on $q$. The stabilizer of $p$ is isomorphic to the group of similarities of the affine chart $\Mink_-(p)$. The diamonds $J(p,q)$ and $J(p,q')$ can be seen as the futures of two degenerate affine hyperplanes, $H$ and $H'$, in the affine chart $\Mink_-(p)$. Since $O(1,n)$ acts transitively on the isotropic directions of $\Mink_-(p)$, there exists an affine isometry $f$ of $\Mink_-(p)$ that sends $H$ on $H'$. Thus, $f$ defines a conformal diffeomorphism between $J(p,q)$ and $J(p,q')$.
\end{enumerate}

Diamonds that contain neither conjugate points nor complete lightlike geodesics in their interior can be realized as diamonds in some affine chart. These are described in the thesis of A. Chalumeau (see \cite[Sec. 1.4]{Chalumeau2024}). The author shows that they have all the same conformal structure, equivalent to $\hy^{n-1} \times \R$ equipped with the conformal class of the Lorentzian metric $g_{\hy^{n-1}} - dt^2$, where $g_{\hy^{n-1}}$ denotes the hyperbolic metric and $dt^2$ the usual metric on $\R$. 

One way to produce diamonds of Minkowski spacetime in the universal Einstein universe $\eeu$ is by considering a conformal $(n-2)$-sphere $\mathbb{S}^{n-2}$ in $\eeu$. More precisely, the set of points in $\eeu$ which are not causally related to this sphere is the union of two connected components, each one corresponding to an open diamond in Minkowski spacetime. Indeed, given a conformal $(n-1)$-sphere $\mathbb{S}^{n-1}$ containg $\mathbb{S}^{n-2}$, the sphere $\mathbb{S}^{n-2}$ disconnects $\mathbb{S}^{n-1}$ in two disjoint open balls $B$ and $B'$. The stereographic projection with respect to a point $p \in B'$ sends the sphere $\mathbb{S}^{n-1}$ on an Euclidean hyperplane in the affine chart $\Mink_0(p)$, containing the ball $B$. It is then easy to see that the set of points in $\Mink_0(p)$ which are not causally related to $\mathbb{S}^{n-2} = \partial B$ is the union of two disjoint connected components:
\begin{itemize}
\item a bounded component, corresponding to the Cauchy development of $B$ in $\Mink_0(p)$ \textemdash which is a diamond of $\Mink_0(p)$, denoted by $D$;
\item a non-bounded component, corresponding to set of points in $\Mink_0(p)$ which are not causally-related to the closure of $B$.
\end{itemize}
By symmetry, the stereographic projection with respect to a point $p' \in B$ shows that this last component corresponds to the Cauchy development of $B'$ in $\Mink_0(p')$ \textemdash which is a diamond in $\Mink_0(p')$, denoted by $D'$.\\

Last but not least, the diamonds that contain conjugate points in their interior are homeomorphic to a cylinder $\mathbb{S}^{n-1} \times \R$. They do not all share the same conformal structure. Indeed, let $J(p,q)$ and $J(p',q')$ be two such diamonds. Up to conformal diffeomorphism, we can assume that $p' = p$. Then, the diamonds $J(p,q)$ and $J(p,q')$ are conformally equivalent if and only if $q$ and $q'$ are in the same orbit under the action of the stabilizer of $p$ on $\eeu$.

\subsection{Intersection of diamonds in the universal Einstein universe} \label{sec: intersection of diamonds}

In Minkowski spacetime, the intersection of two diamonds is always connected, as diamonds are convex. However, in the universal Einstein universe $\eeu$, the situation is more subtle since there is no notion of convexity anymore. In this section, we present an example of two diamonds in $\eeu$ whose intersection is non-empty and disconnected.

To construct such an example, we consider two diamonds, $D$ and $D'$, in $\eeu$, each conformally equivalent to $\hy^{n-1} \times \R$, satisfying the following conditions:
\begin{enumerate}
\item $D$ and $D'$ are not contained in the same affine chart;
\item the initial intersection of $D$ and $D'$ is non-empty and connected.
\end{enumerate}
We then deform $D$ such that its intersection with $D'$ remains non-empty but becomes disconnected.

To put this into practice, let us fix an affine chart in $\eeu$, identified with Minkowski spacetime $\R^{1,n-1}$. Let $(x,y,z)$ denote a coordinate system, where $x = (x_1, \ldots, x_{n-2}) \in \R^{n-2}$ and $y,z \in \R$, such that the quadratic form $q_{1,n-1}$ of Minkowski spacetime is expressed by
\begin{align*}
q_{1,n-1}(x,y,z) &= x_1^2 + \ldots + x_{n-2}^2 + y^2 - z^2.
\end{align*}
Let $D$ be the diamond $I(p,q)$ of $\R^{1,n-1}$ where $p = (0,0,1)$ and $q = (0,0,-1)$. It intersects the hyperplane $z = 0$ in an open ball $B$ centered in $O = (0,0,0)$. Let $B'$ be an open ball in this hyperplane, also centered in $O$, with radius less than that of $B$. Let $D'$ be the diamond of $\eeu$ whose intersection with the affine chart $\R^{1,n-1}$ is exactly the set of points which are not causally related to the closure of $B'$. It is clear that the diamonds $D$ and $D'$ have a non-empty and connected intersection. In particular, the intersection of $D$, $D'$ and the hyperplane $z = 0$ is the annulus $B \backslash \overline{B'}$ (see Figure \ref{fig: intersection of diamonds}).

\begin{figure}[h!] 
\centering
\includegraphics[scale=1.5]{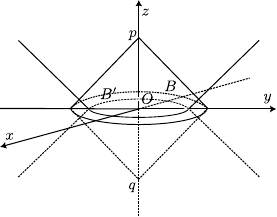}
\caption{Intersection of two diamonds in $\widetilde{Ein}_{1,2}$.}
\label{fig: intersection of diamonds}
\end{figure}

Now, we deform $D$ by applying a loxodromic element $\gamma$ of $O(1,n-1)$ defined, in the coordinate system $(x,a,b)$ where $a = y + z$ and $b = y - z$, by the matrix
\begin{align*}
\left(\begin{array}{ccc}
I_{n-2} &              & \\
        & \lambda^{-1} &  \\
        &              & \lambda 
\end{array} \right)
\end{align*}
where $\lambda > 1$ and $I_{n-2}$ denotes the identity matrix of size $(n-2)$. 

Note that since the $(n-2)$-plane $\{y = z = 0\}$ is preserved by $\gamma_k$, its intersection with $\gamma_k.D \cap D'$ is equal to its intersection with $D \cap D'$. This shows that the intersection $\gamma_k.D \cap D'$ is non-empty for every integer $k$. 

Now, we show that the timelike hyperplane
\begin{align*}
H_1 := \{x_1 = 0\}
\end{align*}
disconnects the intersection $\gamma_k.D \cap D'$ when $k$ is big enough. 

On the one hand, the intersection of $D$ with the hyperplane $H_1$ is precisely the diamond $I(p,q)$ in this hyperplane. On the other hand, the intersection of $D'$ with $H_1$ is the set of points which are not causally related to the ball defined by the intersection of the ball $B'$ with $H_1$. In dimension $3$, the intersection of $D$, $D'$ and the hyperplane $H_1$ is the disjoint union of two diamond shaped regions (see Figure \ref{fig: coupe verticale}).

\begin{figure}[h!]
\centering
\includegraphics[scale=1.5]{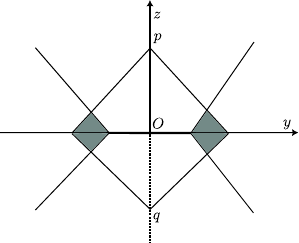}
\caption{Intersection of $D$ and $D'$ in the hyperplane $H_1$ in dimension $3$.}
\label{fig: coupe verticale}
\end{figure}

Let $\Delta_+$ and $\Delta_-$ be the lines spanned respectively by $(0,-1,1)$ and $(0,1,1)$ in the coordinate system $(x,y,z)$. They correspond to the attracting and the repulsing isotropic directions of the sequence $(\gamma_k)_k$. When $k$ tends to $+\infty$, the diamonds $\gamma_k.D$ accumulates on the line $\Delta_+$. Therefore, when $k$ is big enough, the intersection of $\gamma_k.D$, $D'$ and the hyperplane $H_1$ is empty (see Figure \ref{fig: coupe verticale post-deformation}). It follows that the hyperplane $H_1$ disconnects $\gamma_k.D \cap D'$ when $k$ is big enough. Hence, $\gamma_k.D$ and $D'$ are two diamonds of $\eeu$ whose intersection is non-empty and disconnected.

\begin{figure}[h!]
\centering
\includegraphics[scale=1.5]{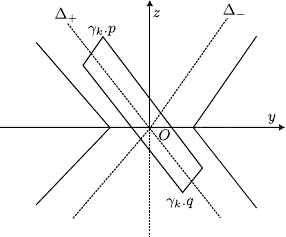}
\caption{The diamonds $\gamma_k.D$ and $D'$ in the hyperplane $H_1$ in dimension $3$.}
\label{fig: coupe verticale post-deformation}
\end{figure}

\subsection{Penrose boundary}

In this section, we describe the conformal boundary of Minkowski spacetime introduced by R. Penrose \cite{Penrose}. It is conformally equivalent to the boundary of an affine chart in Einstein universe.

\begin{definition}
Let $M(\mathrm{x})$ be an affine chart of $\mathsf{Ein}_{1,n-1}$. The regular part of the lightcone of $\mathrm{x}$ is called \emph{Penrose boundary} of $M(\mathrm{x})$ and is denoted $\mathcal{J}(\mathrm{x})$.
\end{definition}

The regular part of the lightcone of $\mathrm{x} \in \mathsf{Ein}_{1,n-1}$ is simply the lightcone of $\mathrm{x}$ minus the singular point $\mathrm{x}$, it is topologically a cylinder. In the double cover $Ein_{1,n-1}$, the Penrose boundary of an affine chart $M(\mathrm{x})$ refers to the regular part of the lightcone of $\mathrm{x} \in Ein_{1,n-1}$ which is the complement in the lightcone of $\mathrm{x}$ of the two singular points $\mathrm{x}$ and $-\mathrm{x}$. This is the union of two connected components $\mathcal{J}^+(\mathrm{x})$ and $\mathcal{J}^-(\mathrm{x})$ where 
\begin{itemize}
\item $\mathcal{J}^+(\mathrm{x})$ fibers trivially over the sphere $\mathrm{S}^+(\mathrm{x})$ of future lightlike directions at $\mathrm{x}$;
\item $\mathcal{J}^-(\mathrm{x})$ fibers trivially over the sphere $\mathrm{S}^-(\mathrm{x})$ of past lightlike directions at $\mathrm{x}$.
\end{itemize}
The fiber over a direction $[v] \in \mathrm{S}^\pm(\mathrm{x})$ is the lightlike geodesic contained in $\mathcal{J}^\pm(\mathrm{x})$ tangent to $v$ at $\mathrm{x}$.

\begin{figure}[h!]
\centering
\includegraphics[scale=1.4]{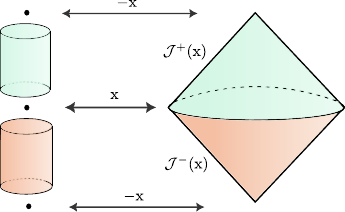}
\caption{Penrose boundary of an affine chart $M(\mathrm{x})$ in $Ein_{1,2}$.}
\end{figure}

In the universal cover $\eeu$, the Penrose boundary of an affine chart $\Mink_0(p)$ is the disjoint union of the regular parts of $\partial I^+(p)$ and $\partial I^-(p)$, denoted $\mathcal{J}^+(p)$ and $\mathcal{J}^-(p)$ respectively.

\paragraph{Penrose boundary and degenerate affine hyperplanes of an affine chart.} Let $M(\mathrm{x})$ be an affine chart of $Ein_{1,n-1}$. We prove that each connected component, $\mathcal{J}^+(\mathrm{x})$ and $\mathcal{J}^-(\mathrm{x})$, of the Penrose boundary is in bijection with the space of degenerate affine hyperplanes of $M(\mathrm{x})$. We write it for $\mathcal{J}^+(\mathrm{x})$ but of course it is similar for $\mathcal{J}^-(\mathrm{x})$.

\begin{lemma}
The intersection of the affine chart $M(\mathrm{x})$ with the lightcone of a point $\mathrm{y} \in \mathcal{J}^+(\mathrm{x})$ is a degenerate affine hyperplane of $M(\mathrm{x})$.
\end{lemma}

\begin{proof}
Let $x, y \in \R^{2,n}$ be two representants of $\mathrm{x}$ and $\mathrm{y}$ respectively. On the one hand, $<y,x>_{2,n} = 0$, i.e. $\vect(y) \subset x^\perp$. On the other hand, since $\mathrm{y} \not \in \{\mathrm{x}, -\mathrm{x}\}$, the lightlike line $\vect(y)$ is transverse to $\vect(x)$. Therefore, the projection $[y]$ of $y$ in the quotient $x^\perp/\vect(x)$ is a non-trivial isotropic vector. It is easy to check that the map $f_x$ defined above induces on the intersection of $M(\mathrm{x})$ with the lightcone of $\mathrm{y}$ an affine structure of direction the orthogonal of $[y]$ in $x^\perp/\vect(x)$. The lemma follows.
\end{proof}

\begin{lemma} \label{lemma: Penrose boundary and lightlike directions of Minkowski spacetime}
Every lightlike geodesic of the Penrose boundary $\mathcal{J}^+(\mathrm{x})$ defines a unique lightlike direction of the affine chart $M(\mathrm{x})$ and vice versa.
\end{lemma}

\begin{proof}
Let $x \in \R^{2,n}$ be a representant of $\mathrm{x}$. Recall that the vector space associated to the affine chart $M(\mathrm{x})$ is  $x^\perp/\vect(x)$. A lightlike geodesic of $\mathcal{J}^+(\mathrm{x})$ is the intersection of $Ein_{1,n-1}$ with the projectivization of a totally isotropic $2$-plane containing $x$. Therefore, a lightlike geodesic of $\mathcal{J}^+(\mathrm{x})$ is equivalent to the data of an isotropic vector in $x^\perp$ transverse to $\vect(x)$, in other words the data of an isotropic vector of $x^\perp/\vect(x)$.
\end{proof}

\begin{lemma}\label{lemma: section of Penrose boundary}
The intersection of the Penrose boundary $\mathcal{J}^+(\mathrm{x})$ with the lightcone of a point of the affine chart $M(\mathrm{x})$ is a section of the trivial fiber bundle $\mathcal{J}^+(\mathrm{x}) \to \mathrm{S}^+(\mathrm{x})$.
\end{lemma}

\begin{proof}
Let $\mathrm{x}_0 \in M(\mathrm{x})$. Let $x, x_0 \in \R^{2,n}$ be two representant of $\mathrm{x}$ and $\mathrm{x}_0$ respectively. The intersection of $\mathcal{J}^+(\mathrm{x})$ with the lightcone of $\mathrm{x}_0$ is a connected component of the intersection of $Ein_{1,n-1}$ with the projectivization of $x^\perp \cap x_0^\perp$. Notice that $x^\perp \cap x_0^\perp = \vect(x,x_0)^\perp$. Since $<x,x_0>_{2,n} < 0$, the subspace $\vect(x,x_0)$ is of type $(1,1)$. Then, $\vect(v,v_0)^\perp$ is of type $(1,n-1)$. It follows that the intersection of $\mathcal{J}^+(\mathrm{x})$ with the lightcone of $\mathrm{x}_0$ is a conformal $(n-2)$-sphere that meets every lightlike geodesic of $\mathcal{J}^+(\mathrm{x})$. The lemma follows.
\end{proof}

Let $f$ be the map which associates to every point $\mathrm{y} \in \mathcal{J}^+(\mathrm{x})$ the intersection of the lightcone of $\mathrm{y}$ with the affine chart $M(\mathrm{x})$.

\begin{proposition}\label{prop: Penrose bound. and degenerate hyperplans}
The map $f$ is a natural bijection between $\mathcal{J}^+(\mathrm{x})$ and the space of degenerate affine hyperplanes of the affine chart $M(\mathrm{x})$.
\end{proposition}

\begin{proof}
We construct the inverse of $L$. Let $P$ be a degenerate affine hyperplane of $M(\mathrm{x})$. It is directed by the orthogonal of a lightlike direction of $M(\mathrm{x})$. By Lemma \ref{lemma: Penrose boundary and lightlike directions of Minkowski spacetime}, to this lightlike direction corresponds a unique lightlike geodesic $\varphi$ of $\mathcal{J}^+(\mathrm{x})$. Let $\mathrm{x}_0 \in P$. By Lemma \ref{lemma: section of Penrose boundary}, the intersection of the lightcone of $\mathrm{x}_0$ with $\mathcal{J}^+(\mathrm{x})$ meets every lightlike geodesic of $\mathcal{J}^+(\mathrm{x})$ in a unique point, in particular it meets $\varphi$ in a unique point $\mathrm{p}$. We call $g$ the map which sends $P$ on $\mathrm{p}$. It is easy to check that $g = f^{-1}$.
\end{proof}

\begin{remark}\label{remark: Penrose boundary in the universal cover}
Let $p \in \eeu$. Given a point $q \in \mathcal{J}^+(p)$, the intersection of the past lightcone of $q$ with the affine chart $\Mink_0(p)$ is a degenerate affine hyperplane $H(q)$. Therefore, 
\begin{itemize}
\item the intersection of $I^-(q)$ with $\Mink_0(p)$ is the chronological past of $H(q)$ in $\Mink_0(p)$;
\item the complement of $I^-(p)$ in $\Mink_0(p)$ is the chronological future of $H(q)$ in $\Mink_0(p)$.
\end{itemize}
Notice that $I^+(q)$ is disjoint from $\Mink_0(p)$. Hence, the chronological future of $H(q)$ is exactly the set of points of $\Mink_0(p)$ which are not causally related to $q$.
\end{remark}

\paragraph{Cylindrical model.}  In this paragraph, we precise the identification between the space of degenerate affine hyperplanes of Minkowski spacetime $\M^n$ and the cylinder $\mathbb{S}^{n-2} \times \R$. This identification requires the choice of:
\begin{enumerate}
\item an origin $p_0$ in the affine space $\M^n$;
\item a unit timelike vector $v_0$ in the underlying vector space $V$ of $\M^n$.
\end{enumerate}
We denote by $<.,.>_{1,n-1}$ the quadratic form of signature $(1,n-1)$ on $V$. The intersection of the lightcone of $V$ with the linear hyperplane $<.,v_0>_{1,n-1} = -1$ is a conformal sphere of dimension $(n-2)$, denoted $\mathcal{S}(v_0)$.
The map which associate to every couple $(v,s) \in \mathcal{S}(v_0) \times \R$ the degenerate affine hyperplane $-<. - p_0, v>_{1,n-1} = s$ defines a one-to-one parametrization of the space of degenerate affine hyperplanes of $\M^n$. Let us make a few remarks: 
\begin{enumerate}
\item The degenerate affine hyperplane $-<. - p_0, v>_{1,n-1} = s$ is the translation of the hyperplane directed by the orthogonal of $v$ going through $p_0$ by the vector $sv_0$.
\item Two points $(v,s)$ and $(v,s')$ of $\mathcal{S}(v_0) \times \R$ on the same vertical line correspond to \emph{parallel} degenerate affine hyperplanes; the hyperplane $(v,s)$ is in the chronological future of the hyperplane $(v,s')$ if and only if $s > s'$.
\item A section of the cylinder $\mathcal{S}(v_0) \times \R$ corresponds to a point of $\M^n$. More precisely, the intersection of all the degenerate affine hyperplanes defined by the points of this section is reduced to a point of $\M^n$.
\end{enumerate}

\subsection{Intersections of conformal spheres with an affine charts}

% 1- Einstein universe: definition, lightlike geodesics, lightcone, conformal spheres
% 2- Affine charts
% 3- Penrose boundary
% 4- Conformal spheres and lightcones in an affine chart

We describe here the intersection of a conformal sphere with an affine chart $M(\mathrm{x})$ of $Ein_{1,n-1}$.

\begin{lemma}\label{lemma: euclidean hyperplanes in an affine chart}
The intersection of the affine chart $M(\mathrm{x})$ with a conformal $(k - 1)$-sphere of $Ein_{1,n-1}$ going through $\mathrm{x}$ is a spacelike $(k - 1)$-plane of $M(\mathrm{x})$.
\end{lemma}

\begin{proof}
Let $\mathrm{S}$ be a conformal $(k - 1)$-sphere going through $\mathrm{x}$. It is the intersection of $Ein_{1,n-1}$ with the projectivization of a Lorentzian $(k + 1)$-plane $P$ of $\R^{2,n}$ containing $\vect(x)$ where $x \in \R^{2,n}$ is a representant of $\mathrm{x}$. It is easy to check that the restriction of $f_x$ to $\mathrm{S} \cap M(\mathrm{x})$ defines an affine structure with direction $x^{\perp_P} / \vect(x)$ where $x^{\perp_P}$ denotes the orthogonal of $x$ in $P$. Since $P$ is Lorentzian, the restriction  of $q_{2,n}$ to $P$ induces a positive definite quadratic form on $x^{\perp_P} / \vect(x)$. The lemma follows.
\end{proof}

\begin{lemma}\label{lemma: hyperboloid in an affine chart}
The intersection of the affine chart $M(\mathrm{x})$ with a conformal $(k-1)$-sphere of $Ein_{1,n-1}$ avoiding $\mathrm{x}$ is one sheet of a two-sheeted hyperboloid of dimension $(k-1)$ in $M(\mathrm{x})$.
\end{lemma}

\begin{proof}
Let $S$ be a conformal $(k - 1)$-sphere avoiding $\mathrm{x}$. It is the intersection of $Ein_{1,n-1}$ with the projectivization of a subspace of a Lorentzian $(k+1)$-plane $P$ of $\R^{2,n}$ which does not contain $\vect(x)$ where $x \in \R^{2,n}$ is a representant of $\mathrm{x}$. 

Set $Q = P \oplus \vect(x)$. Notice that this sum is not orthogonal. Indeed, otherwise $S$ would be in the lightcone of $\mathrm{x}$, hence disjoint from $M(\mathrm{x})$. It follows that $Q$ is a subspace of $\R^{2,n}$ of type $(2,k+1)$. The orthogonal of $P$ in $Q$ is then a timelike line $\vect(y)$ where $y \in \R^{2,n}$ denotes a unit vector director. Let $x' \in \vect(x)$ such that $<x',y>_{2,n} = -1/2$. We call $s(x')$ the symmetric of $x'$ with respect to $\vect(y)$, i.e. $s(x') = y - x'$. We have $<s(x'), s(x')>_{2,n} = 0$ and $<s(x'), x'>_{2,n} = -1/2$. Hence, $[s(x')]$ is a point of the intersection of $M(\mathrm{x})$ and the projectivization of $Q$. A similar proof than that of Lemma \ref{lemma: euclidean hyperplanes in an affine chart} shows that this last intersection is a Lorentzian $(k+2)$-plane, denoted $\mathbb{M}^{1,k+1}$. 

Set $\mathrm{z}_0 := [s(x')] \in \mathbb{M}^{1,k+1}$. We show that the intersection $M(\mathrm{x}) \cap S$ is exactly the set of points $\mathrm{z} \in \M^{1,k+1}$ such that the vector $\mathrm{z} - \mathrm{z}_0$ is a unit timelike vector. Let $\mathrm{z} \in \mathbb{M}^{1,k+1}$ and let $z \in Q$ be a representant of $\mathrm{z}$ such that $<z,x>_{2,n} = -1/2$. By definition, the norm of the vector $\mathrm{z} - \mathrm{z}_0$ is given by
\begin{align*}
<z - s(x'), z - s(x')>_{2,n} = - 2 <z, s(x')>_{2,n} \\
                             = - 2 <z, y - x'>_{2,n} \\
                             = - 2 <z,y>_{2,n} - 1.
\end{align*}
Hence, the norm of $\mathrm{z} - \mathrm{z}_0$ equals $-1$ if and only if $<z,y>_{2,n} = 0$, which is that $z \in P$ or, equivalently, that $\mathrm{z} \in S \cap M(\mathrm{x})$.
\end{proof}

\section{Regular domains} \label{sec: regular domains}

% Definition + characterization with shadows 

% les domaines réguliers futurs sont exactement les ouverts causalement convexes maximaux futurs complets de Minkowski.

Penrose boundary is closely related to the notion of \emph{regular domains} of Minkowski spacetime. 

\begin{definition} \label{def: regular domains}
\emph{A future-regular domain} is a non-empty convex open domain of Minkowski spacetime which can be described as the intersection of strict future half-spaces bounded by a degenerate hyperplane.
\end{definition}

\begin{remark}
We define similarly past-regular domains by reversing the time-orientation.
\end{remark}

Usually, the definition of regular domains excludes half-spaces bounded by a degenerate affine hyperplane and the Minkowski spacetime itself (see e.g. \cite[def. 4.5]{barbot} or \cite[def. 4.1]{bonsante2005flat}). However, for the sake of convenience, we choose to include them in the definition. We distinguish these \mbox{examples} from the regular domains defined by at least two non-parallel degenerate affine hyperplane by using the term \emph{proper}:

\begin{definition}
A future-regular (or past-regular) domain of Minkowski spacetime defined by at least two non-parallel degenerate hyperplanes is said to be \emph{proper}.
\end{definition}

Let $\Lambda \subset \mathcal{J}^+(p)$. It naturally defines a convex domain in both affine charts $\Mink_0(p)$ and $\Mink_+(p)$: 
\begin{itemize}
\item the set $\Omega^+(\Lambda)$ of points of $\Mink_0(p)$ which are not causally related to any point of~$\Lambda$;
\item the set $\Omega^-(\Lambda)$ of points of $\Mink_+(p)$ which are not causally related to any point of $\Lambda$.
\end{itemize}
The set $\Omega^+(\Lambda)$ corresponds to the intersection of the strict future half-spaces of $\Mink_0(p)$ bounded by a degenerate hyperplane of $\Lambda$ and so it is a \emph{future} convex set (see Remark~\ref{remark: Penrose boundary in the universal cover}). Similarly, $\Omega^-(\Lambda)$ is the intersection of the strict past half-spaces of $\Mink_+(p)$ bounded by a degenerate hyperplane of $\Lambda$ and so it is a \emph{past} convex set. Indeed, $\mathcal{J}^+(p) = \mathcal{J}^-(\sigma(p))$ is the past Penrose boundary of the affine chart $\Mink_0(\sigma(p)) = \Mink_+(p)$. Hence, $\Omega^+(\Lambda)$ and $\Omega^-(\Lambda)$ are regular domains if and only if they are non-empty and open. Notice that if $\Lambda$ is closed, then $\Omega^\pm(\Lambda)$ is open (maybe empty). In Section \ref{sec: regular domains and shadows}, we provide a sufficient and necessary condition for $\Omega^\pm(\Lambda)$ to be non-empty.

%\begin{remark}
%Set $\mathrm{x} := \pi(p)$. The restriction of $\pi$ to $\Omega^+(\Lambda)$ is injective (since it is contained in the affine chart $\Mink_0(p)$); moreover, by \cite[Lemma 10.13]{andersson2012}, its image is the future convex domain of $M(\mathrm{x})$ defined as 
%\begin{align*}
%\pi(\Omega^+(\Lambda)) &= \{\mathrm{y} \in M(\mathrm{x}):\ <\mathrm{y},\mathrm{y}_0>_{2,n} < 0,\ \forall \mathrm{y}_0 \in \pi(\Lambda)\}.
%\end{align*}
%Similarly, the restriction of $\pi$ to $\Omega^-(\Lambda)$ is injective and its image is the past convex domain of $M(-\mathrm{x})$ defined as
%\begin{align*}
%\pi(\Omega^-(\Lambda)) &= \{\mathrm{y} \in M(-\mathrm{x}):\ <\mathrm{y},\mathrm{y}_0>_{2,n} < 0,\ \forall \mathrm{y}_0 \in \pi(\Lambda)\}.
%\end{align*}
%\end{remark}

\begin{example}[Misner domains]
Proper regular domains defined by exactly two non-parallel affine degenerate hyperplanes are distinguishable: their quotients by a suitable discrete subgroup of affine transformations of Minkowski spacetime define a family of globally hyperbolic Cauchy-compact flat spacetimes, called \emph{Misner spacetimes}. This terminology was introduced in \cite[Sec. 3.2]{barbot}, as these spacetimes generalize the two-dimensional case introduced by Misner in \cite{misner1967taub}. In that example, the spacetime is constructed by taking the quotient of a half space in $\R^{1,1}$ bounded by a lightlike straight line, under a boost of $\R^{1,1}$.
\end{example}

\subsection{Characterization of regular domains by shadows} \label{sec: regular domains and shadows}

Let $\Mink_0(p)$ be an affine chart of $\eeu$ and let $\Lambda$ be a closed subset of $\mathcal{J}^+(p)$. Consider the open future-convex subset $\Omega^+(\Lambda)$ of $\Mink_0(p)$ defined as above, i.e.
\begin{align*}
\Omega^+(\Lambda) &:= \Mink_0(p) \backslash J^-(\Lambda).
\end{align*}
We provide a necessary and sufficient condition for $\Omega^+(\Lambda)$ to be non-empty using the notion of \emph{shadows}, introduced by C. Rossi in her thesis (see \cite[Chap. 4]{salveminithesis}). We recall the definition below.

\begin{definition}
Let $M$ be a spacetime and let $A$ be an achronal topological hypersurface of $M$. For every $p \in M$, we call \emph{shadow of $p$ on $A$} the set of points of $A$ which are causally related to $p$. We denote it $O(p,A)$ or simply $O(p)$ when there is no confusion on the subset $A$.
\end{definition}

\begin{example}
Given a point $q \in \Mink_0(p)$, the shadow of $q$ on $\mathcal{J}^+(p)$ is the intersection of the causal future of $q$ with $\mathcal{J}^+(p)$. After identification of $\mathcal{J}^+(p)$ with the cylinder $\mathbb{S}^{n-2} \times \R$, the intersection of the lightcone of $q$ with $\mathcal{J}^+(p)$ corresponds to an ellipse and the shadow of $q$ on $\mathcal{J}^+(p)$ corresponds to the upper region of the cylinder bounded by this ellipsoid (see Figure \ref{fig: shadow}).
\end{example}

Suppose that $\Omega^+(\Lambda)$ is non-empty. The following proposition characterizes the points of $\Omega^+(\Lambda)$ by their shadows on $\mathcal{J}^+(p)$. The shadow of a point $q$ of $\Mink_0(p)$ on $\mathcal{J}^+(p)$ is denoted $O(q)$.

\begin{proposition}\label{prop: characterization of regular domains by shadows}
The regular domain $\Omega^+(\Lambda)$ is the set of points of $\Mink_0(p)$ whose shadows on the Penrose boundary $\mathcal{J}^+(p)$ are disjoint from $\Lambda$. \qed
\end{proposition}

\begin{proof}
Let $q \in \Omega^+(\Lambda)$. Suppose there exists $r \in O(q) \cap \Lambda$. Then, $r \in J^+(q)$, equivalently $q \in J^-(r)$ with $r \in \Lambda$. Contradiction. Thus, $O(q) \cap \Lambda = \emptyset$.

Let $q \in \Mink_0(p)$ such that $O(q) \cap \Lambda = \emptyset$. Suppose that $q \not \in \Omega^+(\Lambda)$. Then, there exists $r \in \Lambda$ such that $q \in J^-(r)$. Equivalently, $r \in J^+(q)$. Hence, $r \in O(q) \cap \Lambda$. Contradiction. Thus, $q \in \Omega^+(\Lambda)$.
\end{proof}

\begin{figure}[h!]
\centering
\includegraphics[scale=1]{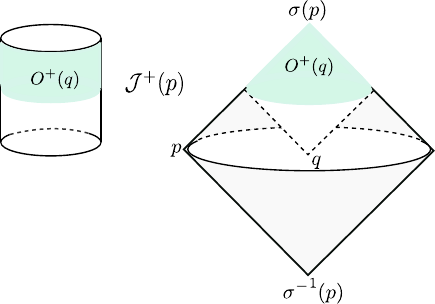}
\caption{Shadow of a point $q \in \Mink_0(p)$ on the Penrose boundary $\mathcal{J}^+(p)$ in dimension~$3$.}
\label{fig: shadow}
\end{figure}

An immediate consequence of Proposition \ref{prop: characterization of regular domains by shadows} is that it provides a criterion for $\Omega^+(\Lambda)$ to be non-empty, in other words to be a future-regular domain:

\begin{corollary} \label{cor: criterion regular}
The domain $\Omega^+(\Lambda)$ is non-empty if and only if there exists a point in $\Mink_0(p)$ whose shadow on $\mathcal{J}^+(p)$ is disjoint from $\Lambda$.
\end{corollary}

After identification of $\mathcal{J}^+(p)$ with $\mathbb{S}^{n-2} \times \R$, this criterion is equivalent to say that there exists a real scalar $C$ such that for every element $(u,s)$ of $\Lambda$, the second component is less than $C$ (see \cite[Proposition 4.10]{barbot}).

\begin{remark}
Proposition \ref{prop: characterization of regular domains by shadows} still holds for the domain $\Omega^-(\Lambda)$ of the affine chart $\Mink_+(p)$. The reformulation of the criterion given by Corollary \ref{cor: criterion regular} in a parametrization $\mathbb{S}^{n-2} \times \R$ is similar to the one stated above after switching \emph{less than} by \emph{greater than}.
\end{remark}

\section{GH conformally flat spacetimes} \label{sec: conformally flat spacetimes}

\subsection{Conformally flat Lorentzian structures}

A spacetime is said to be \emph{conformally flat} if it is locally conformal to Minkowski spacetime. In dimension $n \geq 3$, by Liouville theorem (see Theorem \ref{Liouville theorem}), a spacetime is conformally flat if and only if it is equipped with a $(G,X)$-structure where $X = \eeu$ and \mbox{$G = \Conf(\eeu)$} is its group of conformal transformation.
Therefore, a conformally flat Lorentzian structure on a manifold $M$ of dimension $n \geq 3$ is encoded by the data of a development pair $(D,\rho)$ where $D: \tilde{M} \to \eeu$ is a local diffeomorphism called \emph{developing map} and $\rho: \pi_1(M) \to \Conf(\eeu)$ is the associated \emph{holonomy morphism}. We direct the reader not familiar with $(G,X)$-structures to \cite[Chapter 5]{goldman}. 

Let us make some comments and introduce some vocabulary:
\begin{enumerate}
\item In general, a developing map is only a local diffeomorphism, neither injective nor surjective. When $D$ is a global diffeomorphism, we say that the conformally flat Lorentzian structure on $M$ is \emph{complete}. 
\item A conformally flat spacetime $M$ is said to be \emph{developable} if any developing map descends to the quotient, giving a local diffeomorphism from $M$ to $\eeu$.
\item Two points $p, q$ of a developable conformally flat spacetime $M$ are said to be \emph{conjugate} if their images under a developing map are conjugate in $\eeu$. 
\item C. Rossi proved in \cite[Theorem 10]{Salvemini2013Maximal} that the existence of conjugate points in a GH developable conformally flat spacetime is a sufficient condition of completeness. We provide a shorter and simpler proof of this result in Section \ref{sec: canonical neighborhoods}.
\end{enumerate}

\subsection{GH developable conformally flat spacetimes}

Let $M$ be a GH developable conformally flat spacetime and let $D: M \to \eeu$ be a developing map. We suppose that $M$ is not complete. 

In this section, we give a criterion for a \textbf{causally convex} open subset $U$ of $M$ to be a domain of injectivity of the developing map $D$. First, we recall the definition of the causal convexity.

\begin{definition}
A subset $U$ of a spacetime $V$ is said to be \emph{causally convex} if any causal curve of $V$ connecting two points $p,q $ of $U$ is contained in $U$; equivalently, if the diamond $J(p,q)$ is contained in $U$. 
\end{definition}

An important property of causally convex open subsets is that \emph{they are globally hyperbolic if the ambient spacetime is globally hyperbolic}. Indeed, the diamonds of such a subset $U$ are exactly the diamonds of the ambient spacetime contained in $U$, and thus they are compact.

\begin{lemma}[Criterion of injectivity] \label{lemma: nc injectivity}
Let $U$ be a causally convex open subset of $M$ and let $\Sigma$ be a Cauchy hypersurface of $U$ such that its image $D(\Sigma)$ is achronal in $\eeu$.\\ Then, if the developing map $D$ is injective on $\Sigma$, it is injective on $U$. 
\end{lemma}

\begin{proof}
Let $p,q \in U$ such that $D(p) = D(q)$. We consider a spatio-temporal decomposition $\mathbb{S}^{n-1} \times \R$ of $\eeu$. We call $T$ the vector field on $M$ defined as the pull-back by $D$ of the vector field $\partial_t$ on $\eeu \simeq \mathbb{S}^{n-1} \times \R$. The flow of $T$ defines a foliation of $M$ by smooth inextendible timelike curves. Let $\tilde{\gamma}_p$ (resp. $\tilde{\gamma}_q$) be the leaf going through $p$ (resp. $q$). Since $U$ is causally convex, the intersection of $\tilde{\gamma}_p$ (resp. $\tilde{\gamma}_q$) with $U$ is connected, we denote it by $\gamma_p$ (resp. $\gamma_q$). Since $\gamma_p$ (resp. $\gamma_q$) is an inextendible timelike curve of $U$, it meets $\Sigma$ in a unique point $p_0$ (resp. $q_0$).

By definition, the image $D(\gamma_p)$ (resp. $D(\gamma_q)$) is contained in the timelike line of the decomposition $\mathbb{S}^n \times \R$ going through $D(p)$ (resp. $D(q)$). Since $D(p) = D(q)$, the images $D(\gamma_p)$ and $D(\gamma_p)$ are contained in the same timelike line $\Delta$. Thus, $D(p_0)$ and $D(q_0)$ belong to the intersection of $\Delta$ with $D(\Sigma)$. Since $D(\Sigma)$ is achronal in $\eu$, this intersection is reduced to a single point, hence $D(p_0) = D(q_0)$. Since $D$ is injective on $\Sigma$, it follows that $p_0 = q_0$. Therefore, the leaves $\tilde{\gamma}_p$ and $\tilde{\gamma}_q$ coincide. By Lemma \ref{lemma: dev sends causal curve on causal curve}, we deduce that $p = q$. The lemma follows.
\end{proof}

\subsection{Causal completion of a GH developable conformally flat spacetime} \label{sec: causal completion}

\emph{The causal completion of a spacetime $V$} was introduced by Geroch, Kronheimer and Penrose in 1972 (see \cite{Kronheimer}). Roughly speaking, this construction consists in attaching to $V$ ideal points corresponding to the endpoints at infinity of inextensible causal curves. Those in the future form the \emph{future causal boundary} and those in the past form \emph{the past causal boundary}. This is formalized by the notion of \emph{Indecomposable Past sets} (abbrev. IPs) and, by symmetry, the notion of \emph{Indecomposable Future sets} (abbrev. IFs) of $V$. In \cite{smai2023causal}, we provide a detailed exposition of these concepts in the context of \emph{GH} spacetimes. \\

In \cite{Kronheimer}, the authors showed that the IPs split into two classes: 
\begin{enumerate}
\item the \emph{proper} IPs (abbrev. PIPs), consisting in the chronological pasts of points;
\item and the \emph{terminal} IPs (abbrev. TIPs), consisting in the chronological pasts of future-inextensible causal curves.
\end{enumerate}
The set of PIPs identifies naturally to $V$ while the set of TIPs form the future causal boundary of $V$. 

By symmetry, the PIFs split also into two classes, the PIFs and the TIFs, defined similarly after reversing the time-orientation. The set of PIFs identifies naturally to $V$ while the set of TIFs defines the past causal boundary of $V$. 

The causal completion of $V$ is defined as the union of IPs and IFs, quotiented by the equivalence relation which identifies $I^-(p)$ to $I^+(p)$ for every $p \in V$. It turns out that the causal completion of $V$ is a topological space in which $V$ embeds as a dense open subset. In \cite[Sec. 3.2, Prop. 3.3]{smai2023causal}, we provide a simpler proof of this fact when $V$ is globally hyperbolic. 

\begin{example}
The future causal boundary of Minkowski spacetime, realized as an affine chart $\Mink_0(p)$ of $\eeu$, is the union of $\mathcal{J}^+(p)$ and $\sigma(p)$ and its past causal boundary is the union of $\mathcal{J}^-(p)$ and $\sigma^{-1}(p)$ (see \cite[Example 3.2]{smai2023causal}).
\end{example}

\begin{example}
The description of inextensible causal curves of $\eeu$ shows that the only TIP and the only TIF of $\eeu$ is $\eeu$ itself. Hence, each of the future and the past causal boundary of $\eeu$ is reduced to a single point.
\end{example}

In \cite[Sec. 4]{smai2023causal}, we proved that when $V$ is a GH developable conformally flat spacetime without conjugate points, its causal completion is a topological manifold with boundary homeomorphic to $S \times [0,1]$, where $S$ is a Cauchy hypersurface of $V$. The assumption of no conjugate points ensures that $V$ is not $\eeu$ for which this result obviously does not hold.

\subsubsection*{Extension of the developing map to the causal boundary}

Let $M$ be a developable GH conformally flat spacetime and let $D: M \to \eeu$ be a developing map. We suppose that $M$ is not complete.

We construct a natural extension of $D$ to \emph{the causal completion of $M$}: \\ 

The developing map $D$ induces naturally a map $\hat{D}$ which associate to every IP $U$ of $M$ the IP $\hat{D}(U)$ of $\eeu$ defined by:
\begin{align*}
\hat{D}(U) &:= I^-(D(U)). 
\end{align*}

\begin{lemma}\label{lemma: image of a PIP}
The image under $\hat{D}$ of a PIP of $M$ is a PIP of $\eeu$. More precisely, for every $p \in M$, we have
\begin{align*}
\hat{D}(I^-(p)) &= I^-(D(p)).
\end{align*}
\end{lemma}

\begin{proof}
Set $U = I^-(p)$. We show that $I^-(D(U)) \subset I^-(D(p))$. Let $r \in I^-(D(U))$. Then, there exists $q \in I^-(p)$ such that $r \in I^-(D(q))$. Since $D$ is conformal, $D(q) \in I^-(D(p))$. By transitivity, we deduce that $r \in I^-(D(p))$. Thus, $I^-(D(U)) \subset I^-(D(p))$. 

Conversely, we prove that $I^-(D(p)) \subset I^-(D(U))$. Let $r \in I^-(D(p))$. Let $q \in I^-(p)$. If $D(q) \in J(D(p), r)$, then $r \in J^-(D(q)) \subset I^-(D(p))$. It follows easily that $r \in I^-(D(U))$. Otherwise, $J(D(p), r)$ and $J(D(p), D(q))$ have a non-empty intersection. By Lemma \ref{lemma: D is injective on diamonds}, we deduce that there exists $q' \in J(p,q)$ such that $D(q') \in J(D(p), r)$. Thus, $r \in J^-(D(q')) \subset I^-(D(p))$. Hence, $r \in I^-(D(U))$. The lemma follows. 
\end{proof}

\begin{lemma} \label{lemma: image of a TIP}
The image under $\hat{D}$ of a TIP of $M$ is a PIP of $\eeu$.
\end{lemma}

\begin{proof}
Let $\gamma: [a,b[ \to M$ be a future-inextensible causal curve such that $I^-(\gamma) = U$. The same arguments used in the proof of Lemma \ref{lemma: image of a PIP} show that $\hat{D}(U) = I^-(D(\gamma))$. Suppose that $D(\gamma)$ is future-inextensible in $\eeu$. Then, there exists two points $p,q$ in the curve $\gamma$ such that $p \in J^+(q)$ and $J(D(p), D(q))$ contains conjugate points. By Lemma \ref{lemma: D is injective on diamonds}, we deduce that $J(p,q)$ contains conjugate points. By \cite[Theorem 10]{Salvemini2013Maximal}, we deduce that $M$ is complete. Contradiction. Hence, $D(\gamma)$ admits a future endpoint $p_+$ in $\eeu$. Consequently, $I^-(D(\gamma)) = I^-(p_+)$.
\end{proof}

Let $\partial^+ M$ denote the future causal boundary of $M$. It follows from Lemmas \ref{lemma: image of a PIP} and \ref{lemma: image of a TIP} that $\hat{D}$ induces a map 
$\hat{\mathbf{D}}$ from $M \cup \partial^+ M$ to $\eeu$ such that 
\begin{enumerate}
\item $\hat{\mathbf{D}}$ \emph{coincide with} $D$ on $M$;
\item for every TIP $U$ of $M$, we have
\begin{align*}
\hat{\mathbf{D}} (U) &= \lim_{t \to b} D \circ \gamma (t)
\end{align*}
where $\gamma: [a,b[ \to M$ is a future-inextensible causal curve such that $U = I^-(\gamma)$.
\end{enumerate}

Let $\partial^- M$ denote the past causal boundary of $M$. By symmetry, the map $D$ induces a map $\check{\mathbf{D}}$ from $M \cup \partial^- M$ to $\eeu$ defined similarly as above with the reverse time-orientation.

Now, let $M^\sharp$ be the causal completion of $M$, i.e. the union of $M$, $\partial^+ M$ and $\partial^- M$. The maps $\hat{\mathbf{D}}$ and $\check{\mathbf{D}}$ allow to extend the developing map $D$ in a map $D^\sharp: M^\sharp \to \eeu$ defined as follow:
\begin{align*}
D^\sharp(\xi) &= \left\lbrace \begin{array}{cc}
\hat{\mathbf{D}}(p) & \mathrm{if}\ p \in \partial^+ M\\
D(p)                & \mathrm{if}\ p \in  M \\
\check{\mathbf{D}}(p) & \mathrm{if}\ p \in \partial^- M.
\end{array} \right.
\end{align*}

\section{Domains of injectivity of the developing map of a GH conformally flat spacetime} \label{sec: canonical neighborhoods}

Let $M$ be a developable GH conformally flat spacetime of dimension $n \geq 3$. Consider a developing map $D$ from $M$ to $\eeu$. This section aims to highlight natural domains of injectivity of the developing map. More precisely, we prove that the restriction of the developing map to any IP/IF is injective.\\

First, in Section \ref{sec: diamonds}, we establish that the developing map is injective on \textbf{diamonds}. This fact proves to be to be central in two ways:

\begin{itemize}
\item On the one hand, it allows us to provide, in Section \ref{sec: conjugate points}, a concise and straightforward proof of Rossi's result, which states that \emph{the existence of conjugate points in $M$ is a sufficient condition for the developing map to be injective} (see \cite[Theorem 10]{Salvemini2013Maximal}). 
\item On the other hand, it allows us to prove, in Section \ref{sec: future of points}, that \emph{the developing map is injective on the future of any point in $M$} (and by symmetry, on the past of any point in $M$), using simple arguments. This result was first established by C. Rossi in her thesis \cite[Chap. 6, Sec. 2, Prop. 2.7]{salveminithesis}, but we propose a simpler and more direct proof here, relying on the new argument that diamonds are domains of injectivity.
\end{itemize} 

\subsection{Diamonds are domains of injectivity} \label{sec: diamonds}

Let us start with the following observation. 

\begin{lemma} \label{lemma: dev sends causal curve on causal curve}
The developing map sends injectively every causal (resp. timelike) curve of $M$ on a causal (resp. timelike) curve of $\eeu$.
\end{lemma}

\begin{proof}
Let $\gamma$ be a causal curve of $M$. Since $D$ is conformal, the image of $\gamma$ under $D$ is a causal curve. If there are two distinct points $p, q$ of $\gamma$ such that $D(p) = D(q)$, then $D(\gamma)$ is a closed causal loop of $\eeu$. Contradiction.
\end{proof}

\begin{lemma}\label{lemma: D is injective on diamonds}
The restriction of the developing map to every non-empty diamond $J(p,q)$ of $M$ is a homeomorphism on its image equal to the diamond $J(D(p), D(q))$.
\end{lemma}

\begin{proof}
By Lemma \ref{lemma: dev sends causal curve on causal curve}, $D(J(p,q)) \subset J(D(p), D(q))$. Set $D_{p,q} := D_{|J(p,q)}: J(p,q) \to J(D(p), D(q))$. Since $J(p,q)$ is compact, the map $D_{p,q}$ is proper. In addition, $D_{p,q}$ is a local homeomorphism. It follows that $D_{p,q}$ is a covering. Since $J(D(p), D(q))$ is simply connected, we deduce that $D_{p,q}$ is a homeomorphism.
\end{proof}

\begin{corollary} \label{cor: D is injective on interior of diamonds}
Let $J(p,q)$ be a diamond with a non-empty interior. Then, the image of the interior of $J(p,q)$ under the developing map $D$ is exactly the interior of the diamond $J(D(p), D(q))$ of $\eeu$.
\end{corollary}

\begin{proof}
By Lemma \ref{lemma: dev sends causal curve on causal curve}, $D(I(p,q)) \subset I(D(p), D(q))$. Conversely, let  $\bar{r} \in I(D(p), D(q))$. 

Suppose first that $I(D(p), D(q))$ is contained in an affine chart of $\eeu$. By Lemma \ref{lemma: D is injective on diamonds}, there is a unique $r \in J(p,q)$ such that $D(r) = \bar{r}$. Suppose that $r \not \in I(p,q)$. Then, there is a past-directed lightlike geodesic $\varphi$ joining $p$ to $r$ or $r$ to $q$. Let us deal with the case where $\varphi$ joins $p$ to $r$, the other case is similar. Since $D$ is conformal and preserves the time-orientation, the image of $\varphi$ under $D$ is a past-directed lightlike geodesic of the affine chart $\R^{1,n}$ joining $D(p)$ to $D(r)$. Thus, $D(r)$ is in the boundary of $J(D(p), D(q))$. Contradiction. Hence, $r \in I(p,q)$. 

Suppose now that $I(D(p), D(q))$ is not contained in an affine chart. Since $\bar{r}$ is in the interior of $J(D(p), D(q))$, there exists an open diamond $I(\bar{p}', \bar{q}')$ contained in affine chart such that $\bar{r} \in I(\bar{p}', \bar{q}') \subset J(D(p), D(q))$. By Lemma \ref{lemma: D is injective on diamonds}, there exists $p', q' \in J(p,q)$ such that $D(J(p',q')) = J(\bar{p}', \bar{q}')$. Then, according to the previous paragraph, $\bar{r} = D(r)$ where $r \in I(p',q') \subset I(p,q)$. The corollary follows. 
\end{proof}

\subsection{Conjugate points and injectivity of the developing map} \label{sec: conjugate points}

It was proved in \cite{Salvemini2013Maximal} that the existence of conjugate points in $M$ is a sufficient condition for $D$ to be injective on $M$. We give a new simple proof of this result relying on the fact that diamonds are domains of injectivity.

\begin{proposition} \label{prop: connjugate points}
If $M$ admits conjugate points, then $M$ admits a topological $(n-1)$-sphere as a Cauchy-hypersurface and $D$ is injective on $M$.
\end{proposition}

\begin{proof}
Let $p,q \in M$ be two conjugate points such that $p \in J^+(q)$. Let $p' \in I^+(p)$ and $q' \in I^-(q)$. Then, $J(p,q) \subset I(p', q')$. By Lemma \ref{lemma: D is injective on diamonds}, the developing map $D$ is injective on $I(p',q')$ and $D(I(p',q')) = I(D(p'), D(q'))$. This open diamond of $\eeu$ contains the conjugate points $D(p)$ and $D(q)$, and thus contains an edgeless achronal topological $(n-1)$-sphere $\Sigma$. We denote by $S$ the preimage of $\Sigma$ under $D$ in $I(p',q')$. This is an achronal compact topological hypersurface of $I(p',q')$, and thus of $M$. We deduce that $S$ is a Cauchy hypersurface of $M$ (see e.g. \cite[Annexe B, Cor. B.6.2]{Smai2022}). The proposition follows from Lemma \ref{lemma: nc injectivity}.
\end{proof}

As a consequence, we obtain immediately Rossi's result:

\begin{corollary} \label{cor: conjugate points}
Let $V$ be a non-developable globally hyperbolic maximal conformally flat spacetime. If $\tilde{V}$ admits conjugate points then $V$ is a finite quotient of $\eeu$. \qed
\end{corollary}

By analogy with the Riemannian setting, we state the following definition.

\begin{definition}
A conformally flat spacetime is said to be \emph{elliptic} if it is the quotient of $\eeu$ by a finite subgroup of conformal transformations of $\eeu$.
\end{definition}

\subsection{GH conformally flat spacetimes without conjugate points} \label{sec: future of points}

We suppose now that $M$ does not contain conjugate points. We devote this section to the proof of the following proposition.

\begin{proposition}\label{prop: injectivity on the future of p}
For every $p \in M$, the restriction of the developing map $D$ to the chronological past $I^-(p)$ of $p$ is injective.
\end{proposition}

The key idea is to describe $I^-(p)$ as the union of the open diamonds $I(p,q)$, where $q \in I^-(p)$. Since each one of these diamonds is a domain of injectivity, all we need to check is that the union of two of them is also a domain of injectivity. To do this, we use the following classical lemma.

\begin{lemma}[Lemme des assiettes]\label{lemme des assiettes}
Let $U$ and $V$ be two open subsets of $M$ on which the developing map $D$ in injective. We suppose that the intersection of $U$ and $V$ is non-empty. Then, if the intersection of the images $D(U)$ and $D(V)$ is connected, the developing map $D$ is injective on the union of $U$ and $V$. 
\end{lemma}

\begin{proof}
See e.g. \cite[Chap. 4, Sec. 1, Lemme 4.1.4]{Smai2022}.
\end{proof}

\begin{remark}
Lemma \ref{lemme des assiettes} is valid for any local homeomorphism $f: X \to Y$ between two manifolds $X$ and $Y$.
\end{remark}

\begin{proof}[Proof of Proposition \ref{prop: injectivity on the future of p}]
Let $x, x' \in I^-(p)$ such that $D(x) = D(x')$. 

There exist $q, q' \in I^-(p)$ such that $x \in I(p,q)$ and $x' \in I(p,q')$. We prove that the union of $I(p,q)$ and $I(p,q')$ is a domain of injectivity. By Lemma \ref{lemme des assiettes}, this reduces to prove that the intersection of $D(I(p,q))$ and $D(I(p,q'))$ in $\eeu$ is connected: According to Corollary \ref{cor: D is injective on interior of diamonds}, $D(I(p,q)) = I(D(p), D(q))$ and $D(I(p,q')) = I(D(p), D(q))$. In particular, $D(I(p,q))$ and $D(I(p,q'))$ are contained in $I^-(D(p))$. Since $M$ does not contain conjugate points, they are, in fact, contained in $\Mink_-(D(p))$. More precisely, they correspond to future cones in this affine chart. Thus, their intersection is connected. Therefore, the union of $I(p,q)$ and $I(p,q')$ is a domain of injectivity; and $x = x'$.
\end{proof}

\begin{corollary}\label{cor: D is injective on future and past}
The restriction of the developing map $D$ to the causal future $J^-(p)$ of $p \in M$ is injective.
\end{corollary}

\begin{proof}
Let $x, y \in J^+(p)$ such that $D(x) = D(y)$. Suppose that $x \neq y$. Then, by Proposition \ref{prop: injectivity on the future of p}, either $x \in \partial I^-(p)$ and $y \in I^-(p)$ (or symetrically, $y \in \partial I^-(p)$ and $x \in I^-(p)$), or $x, y \in \partial I^-(p)$.

Suppose that $x \in \partial I^-(p)$ and $y \in I^-(p)$. Then, $p$ and $x$ are connected by a past lightlike geodesic $\varphi$. Since $D$ is conformal, $D(\varphi)$ is a past lightlike geodesic connecting $D(p)$ to $D(x)$. Moreover, $D(\varphi)$ does not contain conjugate points, otherwise $M$ would contain conjugate points, which would contradict our assumption. Hence, $D(x) \in \partial I^-(D(p))$. However, $D(x) = D(y) \in I^-(D(p))$. Contradiction.

Then, $x, y \in \partial I^-(p)$. Hence, there exist two past lightlike geodesics $\alpha$ and $\beta$, connecting $p$ to $x$ and $y$ respectively. Since $x \neq y$, these two geodesics are tangent to two distinct isotropic vectors at $p$. Given that the developing map is a local diffeomorphism, it follows that $D(\alpha)$ and $D(\beta)$ are two distinct past lightlike geodesics connecting $D(p)$ to $D(x) = D(y)$. This shows that $D(p)$ and $D(x)$ are conjugate. Contradiction. 

Therefore, $x = y$. The corollary follows.
\end{proof}

\begin{corollary}
The restriction of the developing map $D$ to any TIP of $M$ is injective. 
\end{corollary}

\begin{proof}
Since any TIP of $M$ is an increasing union of chronological past of points of $M$, the corollary follows immediately from Proposition \ref{prop: injectivity on the future of p}.
\end{proof}

\begin{remark}
By symmetry, Proposition \ref{prop: injectivity on the future of p} stays true when considering the chronological past of $p$. Consequently, the developing map is injective on the causal past of any point in $M$ and on any TIF of $M$.
\end{remark}

\begin{lemma}\label{lemma: the image of the future is causally convex in ein}
The image of the chronological past of a point $p \in M$ under the developing map $D$ is causally convex in $\eeu$.
\end{lemma}

\begin{proof}
Let $\bar{x}, \bar{y}$ be two causally-related points in $D(I^-(p))$. Suppose for instance that $\bar{x} \in I^+(\bar{y})$. By Proposition \ref{prop: injectivity on the future of p}, there exists $x, y \in I^-(p)$ such that $x \in I^+(y)$ and whose images under $D$ are $\bar{x}$ and $\bar{y}$ respectively. Then, by Lemma \ref{lemma: D is injective on diamonds}, we obtain that $J(\bar{x}, \bar{y}) = D(J(x,y)) \subset D(I^-(p))$. The lemma follows.
\end{proof}

\begin{remark}
By symmetry, the image under the developing map of the chronological past of any point of $M$ is causally convex.
\end{remark} 

It follows immediately from Lemma \ref{lemma: the image of the future is causally convex in ein} that the image under $D$ of any TIP (TIF) of $M$ is also causally convex in $\eeu$.\\

To summarize, we have proved the following result in this section:

\begin{proposition}
The restriction of the developing map to any IP/IF of $M$ is injective. Moreover, the image is causally convex in $\eeu$. \qed
\end{proposition}

\section{Maximality of IPs/IFs} \label{sec: maximality of IPs}

Let $M$ be a non-elliptic developable globally hyperbolic maximal conformally flat spacetime of dimension $n \geq 3$. Consider a developing map $D$ from $M$ to $\eeu$. We devote this section to the proof of the following result.

\begin{theorem} \label{thm: maximality of IPs}
Any IP of $M$ is conformally equivalent to a future-regular domain of Minkowski spacetime.
\end{theorem}

This theorem says, in particular, that the IPs of $M$ are maximal as globally hyperbolic conformally flat spacetimes. It was first established by C. Rossi in her thesis \cite[Chap. 6, Sec. 3, Prop 3.6 and Th. 3.9]{salveminithesis}. However, we provide a simple and synthetic proof here. In addition, we prove that the regular domain corresponding to a PIP, i.e. to the chronological past of a point in $M$, satisfies two remarkable properties:
  
\begin{enumerate}
\item it is defined by \emph{an acausal topological $(n-2)$-sphere of the Penrose boundary}; in particular, it is \emph{proper} i.e. it could not be a half space bounded by a degenerate hyperplane;
\item it is "\emph{strictly convex}" in the following sense: it does not contain any spacelike line segment in its boundary.
\end{enumerate}
Of course, we have similar statements for the PIFs by symmetry. It turns out that these two properties are not true for the TIPs and the TIFs.\\

The proof of Theorem \ref{thm: maximality of IPs} relies on the following result, which we established in a previous paper, stating that the maximal extensions of globally hyperbolic conformally flat spacetimes respect inclusion:

\begin{fact}[{\cite[Theorem 2]{smai2023enveloping}}]\label{fact: max. ext. respect inclusion}
Let $V$ be a globally hyperbolic conformally flat spacetime and let $U$ be a causally convex open subset of $V$. Then, the maximal extension of $U$ is conformally equivalent to a causally convex open subset of the maximal extension of $V$.
\end{fact}

\subsection{Proof of Theorem \ref{thm: maximality of IPs}}

Let us start with the following observation.

\begin{lemma}\label{lemma: the past of p is contained in an affine chart}
Let $U \subset M$ be an IP. The image under $D$ of $U$ is contained in the affine chart $\Mink_-(D^\sharp(U))$. Moreover, it is future-complete in this affine chart.
\end{lemma}

\begin{proof}
We prove the lemma when $U$ is a PIP, i.e. the chronological past of a point $p \in M$. The proof is similar when $U$ is a TIP. 

By Lemma \ref{lemma: dev sends causal curve on causal curve}, $D(I^-(p)) \subset I^-(D(p))$. Suppose there exists $q \in I^-(p)$ such that $D(q) \not \in \Mink_-(D(p))$. Then, $D(q) \in I^-(\sigma^{-1}(D(p)))$.
Consequently, the diamond $J(D(p), D(q))$ contains conjugate points and then, by Lemma \ref{lemma: D is injective on diamonds}, so does the diamond $J(p,q)$. Contradiction. Thus, $D(I^-(p)) \subset \Mink_-(D(p))$. 

Now, we prove that $D(I^-(p))$ is future-complete in $\Mink_-(D(p))$. Let $q \in I^-(p)$. The chronological future of $D(q)$ in $\Mink_-(D(p))$ is $I(D(p), D(q))$. By Corollary \ref{cor: D is injective on interior of diamonds}, we have $I(D(p), D(q)) = D(I(p,q))$. Since $I(p,q) \subset I^-(p)$, we get $D(I(p,q)) \subset D(I^-(p))$. Thus, $D(I^-(p))$ is future-complete in $\Mink_-(D(p))$.
\end{proof}

\begin{proof}[Proof of Theorem \ref{thm: maximality of IPs}]
Throughout this proof, all the identifications are up to conformal diffeomorphism. Let $U \subset M$ be an IP. We call $\hat{U}$ the maximal extension of~$U$. 

By Lemmas \ref{lemma: the past of p is contained in an affine chart} and \ref{lemma: the image of the future is causally convex in ein}, $U$ is identified to a future-complete causally convex open subset of Minkowski spacetime $\M^{n}$. Then by Fact \ref{fact: max. ext. respect inclusion}, $\hat{U}$ is also identified to a causally convex open subset of $\M^{n}$ containing $U$. Since $U$ is future-complete in $\M^n$, we deduce that $\hat{U}$ is also future-complete in $\M^{n}$. Hence, $\hat{U}$ is a future regular domain of $\M^{n}$. 
 
It remains to prove that $\hat{U} \subset U$. Let $q \in \hat{U} \subset \M^n$. Since $\hat{U}$ is a Cauchy extension of $U$, any inextensible timelike curve going through $q$ intersects a Cauchy hypersurface in $U$ in a unique point. Thus, $q$ is chronologically related to some point $r \in U \subset \M^n$. If $q$~is in the chronological future of $r$ in $\M^n$, then $q \in U$ since $U$ is future-complete in $\M^{n}$. Suppose now that $q$ is in the chronological past of $r$ in~$\M^n$. By Fact \ref{fact: max. ext. respect inclusion}, $\hat{U}$ can be seen as a causally convex open subset of $M$ containing $U$. Moreover, according to Lemma \ref{lemma: nc injectivity}, the developing map is injective on $\hat{U}$.  Therefore, $q$ and $r$ can be seen as points of~$M$ such that $q \in I^-(r)$. Since $U$ is an IP, $I^-(r) \subset U$. Thus, $q \in U$. The proposition follows. 
\end{proof}

\subsection{Two remarkable properties of the PIPs}

We devote this section to proving the following two properties of the PIPs of $M$.

\begin{property}\label{prop: no spacelike segment for the PIPs}
The chronological past of a point $p \in M$ is conformally equivalent to a future-complete regular domain of Minkowski spacetime which does not contain any spacelike line segment in its boundary.
\end{property} 

\begin{property}\label{property 2}
The chronological past of a point $p \in M$ is conformally equivalent to a future-complete regular domain of Minkowski spacetime $\M^n$ defined by a topological $(n-2)$-sphere of Penrose boundary.
\end{property}

\subsubsection{Proof of Property \ref{prop: no spacelike segment for the PIPs}}

Set $U := D(I^-(p))$. Suppose there is a spacelike line segment in the boundary of $U$ in $\Mink_-(D(p))$. Let $\Delta$ be the line containing this segment. The closure of $\Delta$ in $\eu$ is the union of $\Delta$ and the first conjugate point of $D(p)$ in the past. This is a conformal $n$-sphere denoted by $S$.

Let $q \in I^+(p)$ and set $V := D(I^-(q))$. Then, $U$ is the intersection of $V$ with $I^-(D(p))$. 
Therefore, the segment $\sigma$ is in the boundary of $V$ in $\Mink_-(D(q))$. In particular, $\sigma$ is convex in $\Mink_-(D(q))$. By Lemma \ref{lemma: hyperboloid in an affine chart}, the intersection of $S$ with $\Mink_-(D(q))$ is a connected component of a two-sheeted hyperboloid denoted by $H$. Since, $H$ contains $\sigma$, it is convex in $\Mink_-(D(p))$. Therefore, $H \subset S$ intersects the lighcone of $D(p)$ in two distinct points. Contradiction.

\subsubsection{Proof of Property \ref{property 2}}

Let $p \in M$. We call $\mathcal{J}^-(p)$ the connected component of the Penrose boundary of $\Mink_-(D(p))$ equal to the regular part of $\partial I^-(D(p))$.

\begin{fact}
The image under $D$ of any photon going through $p$ admits a past endpoint in $\mathcal{J}^-(p)$.
\end{fact}

\begin{proof}
If there is a photon $\varphi$ of $M$ going through $p$ such that $D(\varphi)$ has no past endpoint in $\mathcal{J}^-(p)$, then $D(\varphi)$ contains conjugate points. Contradiction.
\end{proof}

\begin{definition}\label{def: sphere associated to the past of a point}
We call $\Lambda^-(p)$ the subset of $\mathcal{J}^-(p)$ consisting in the past endpoints of the images under $D$ of the photons going through $p$.
\end{definition}

Notice that $\Lambda^-(p)$ is in bijection with the sphere $\mathrm{S}^-(p)$ of past lightlike directions at~$p$. We will see that this bijection is actually a homeomorphism.

\begin{proposition}\label{prop: characterization of the past of a point}
The chronological past of $p$ is conformally diffeomorphic to the regular domain of $\Mink_-(D(p))$ defined by $\Lambda^-(p)$.
\end{proposition}

\begin{proof}
Set $U := D(I^-(p))$ and let $\hat{U}$ denote the regular domain of $\Mink_-(D(p))$ defined by $\Lambda^-(p)$. We prove that $\hat{U}$ is a Cauchy extension of $U$. Since $U$ is maximal, we will deduce that $U = \hat{U}$. 

First, we prove that $U \subset \hat{U}$. By Proposition \ref{prop: characterization of regular domains by shadows}, this consists in proving that for every $x \in U$, the shadow of $x$ on $\mathcal{J}^-(p)$, denoted $O(x, \mathcal{J}^-(p))$, is disjoint from $\Lambda(p)$. Let $x \in U$ and let $q \in I^-(p)$ be the unique point such that $x = D(q)$. Notice that $O(x, \mathcal{J}^-(p))$ is the intersection of $J(D(p), D(q))$ with $\mathcal{J}^-(p)$. By Lemma \ref{lemma: D is injective on diamonds}, this intersection is exactly the image under $D$ of $J(p,q) \cap (\partial I^-(p) \backslash \{p\})$. Hence, $O(x, \mathcal{J}^-(p))$ is disjoint from $\Lambda^-(p)$. Thus, $x \in \hat{U}$.

Now, we fix a Cauchy hypersurface $\Sigma$ of $U$. Let $\gamma$ be an inextensible causal curve of $\hat{U}$ going through a point $x \in \hat{U}$. Let $y$ be the accumulation point of $\gamma$ in $\mathcal{J}^-(p)$. By Proposition \ref{prop: characterization of regular domains by shadows}, the point $y$ is in the future of some point $y_0 \in \Lambda^-(p)$. Therefore, there is a unique point $q$ in the lightcone of $p$ such that $D(q) = y$. It follows that $\gamma$ meets $U$ and thus $\Sigma$. We deduce that $\hat{U}$ is a Cauchy extension of $U$.
\end{proof}

\begin{corollary}
The set $\Lambda^-(p)$ is homeomorphic to $\mathbb{S}^{n-1}$.
\end{corollary}

\begin{proof}
Let $q \in I^+(p)$. We have $J^-(p) \subset I^-(q)$. Then, $D(J^-(p))$ is the intersection of $J^-(D(p))$ with the regular domain $D(I^-(q))$ in the affine chart $\Mink_-(D(q))$. It follows that $\Lambda^-(p)$ is the intersection of the lightcone of $D(p)$ with the boundary of $D(I^-(q))$ in $\Mink_-(D(q))$. By Proposition \ref{prop: characterization of the past of a point}, the regular domain $D(I^-(q))$ is proper. Hence, every lightlike line of $\Mink_-(D(q))$ going through $D(p)$ meets the boundary of $D(I^-(q))$ in a unique point (see \cite[Chap. 2, Sec. 3, Prop. 2.3.7]{Smai2022}). The corollary follows.
\end{proof}

Note that Properties \ref{prop: no spacelike segment for the PIPs} and \ref{property 2} does not hold for the TIPs. Indeed, a Misner domain of an affine chart $\Mink_-(p)$ is a conformally flat spacetime $M$ for which $p$ is a TIP. The past of $p$ in $M$ is exactly the Misner domain which does not satisfy any of Properties \ref{prop: no spacelike segment for the PIPs} and \ref{property 2}. We describe other families of examples in \cite[Section 6]{smai2023causal}.

\paragraph{Data availability statement.} Data sharing not applicable to this article as no datasets were generated during the current study.

%\paragraph{Conflict of interest statement.} This work of the Interdisciplinary Thematic Institute IRMIA++, as part of the ITI 2021-2028 program of the University of Strasbourg, CNRS and Inserm, was supported by IdEx Unistra (ANR-10-IDEX-0002), and by SFRI-STRAT’US project (ANR-20-SFRI-0012) under the framework of the French Investments for the Future Program. 

\bibliographystyle{plain}
\bibliography{biblio}

\end{document}